\documentclass[11pt]{article}

\usepackage{geometry}   
\usepackage{authblk}

\usepackage[utf8]{inputenc} 
\usepackage{hyperref}       
\usepackage{url}            
\usepackage{booktabs}       
\usepackage{amsfonts}       
\usepackage{nicefrac}       
\usepackage{microtype}      

\usepackage{natbib}

\usepackage{amsmath,amsfonts,amssymb,amsthm,array}

\usepackage{algorithm}
\usepackage{caption}
\usepackage{subcaption}
\usepackage{algpseudocode}
\usepackage{mdframed} 
\usepackage{thmtools}

\usepackage{booktabs} 
\usepackage{makecell}

\newcommand{\eqdef}{\; { := }\;}

\newcommand{\R}{\mathbb{R}}

\newcommand{\E}{\mathbb{E}}

\newcommand{\I}{\mathcal{I}}

\newcommand{\cD}{\mathcal{D}}

\DeclareMathOperator*{\argmin}{argmin}

\usepackage[colorinlistoftodos,bordercolor=orange,backgroundcolor=orange!20,linecolor=orange,textsize=scriptsize]{todonotes}

\declaretheorem[style=shaded,within=section]{definition}
\declaretheorem[style=shaded,sibling=definition]{theorem}
\declaretheorem[style=shaded,sibling=definition]{proposition}
\declaretheorem[style=shaded,sibling=definition]{assumption}
\declaretheorem[style=shaded,sibling=definition]{corollary}

\declaretheorem[style=shaded,sibling=definition]{lemma}

\title{\bf Convergence Analysis of Muon-type Methods with Inexact LMO in the Degenerate Case}

\date{}

\author[1]{Xun Qian\thanks{email: \texttt{xun.qian@kaust.edu.sa}}}
\author[1]{Peter Richt\'arik\thanks{email: \texttt{peter.richtarik@kaust.edu.sa}}}

\affil[1]{King Abdullah University of Science and Technology, Thuwal, Saudi Arabia}

\begin{document}
	
	\maketitle
	
	\begin{abstract}
		Muon-type methods have demonstrated potentially superior performance over Adam and its variants, and have shown hyperparameter transferability across model sizes when specific norms are chosen for the LMO in deep architectures. However, while the LMO is solved approximately via iterative algorithms in practice, most convergence analyses consider the ideal case where the search direction is the exact solution to the LMO. Recently, the inexact Muon update was analyzed by \citet{shulgin2025beyond}, which reveals a fundamental coupling between the inexactness and the optimal step size and momentum. However, the convergence is guaranteed for the non-degenerate case only, i.e., the smallest positive singular value of the rescaled momentum is assumed to be bounded below by some positive constant when the spectral norm is used. In this work, we investigate Muon-type methods with inexact LMO in the degenerate case, where the smallest positive singular value of the rescaled momentum can approach zero, for the general non-convex case and the star-convex case with weight decay. Novel assumptions are proposed to address the challenges posed by inexact LMO in such degenerate scenarios, and convergence rates are established under the layer-wise $(L^0, L^1)$-smooth assumption for both cases. 
	\end{abstract}

	\section{Introduction}

	In this work we consider the following optimization problem
	\begin{equation}\label{p:p}
		\min_{X \in {\cal S}} \left\{  f(X) \eqdef \E_{\xi \sim \cD} [f_{\xi} (X) ] \right\}, 
	\end{equation}
	where ${\cal S}$ is a finite-dimensional vector space, $f_\xi: {\cal S} \to \R$ are potentially non-convex and non-smooth but continuously differentiable functions that represent the loss of model parameterized by $X \in {\cal S}$ associated with training data point $\xi$, and $\cD$ is the probability distribution of the training data. We assume that $f(\cdot)$ has a lower bound, i.e., $\inf_{X \in {\cal S}} f(X) > -\infty$ throughout the paper.

	As deep learning models and training datasets scale rapidly, efficient optimizers have become increasingly critical. After more than a decade of dominance by Adam \citep{kinga2015method} and its variants \citep{reddi2019convergence,loshchilovdecoupled,xie2024adan}, Muon \citep{jordan6muon} has demonstrated potentially superior performance over these methods. Muon approximately orthogonalizes the momentum via Newton-Schulz iterations \citep{higham2008functions}, which were introduced in \citep{bernstein2024old} as a computationally efficient alternative to approximate SVD. This orthogonalization strategy can be traced back to the stochastic spectral descent (SSD) method \citep{carlson2015stochastic,carlson2015stochastic-2,carlson2015preconditioned}, which performs steepest descent in the spectral norm. 
	
	The convergence of Muon has been established in several works \citep{li2025note,pethick2025training,kovalev2025understanding,shen2025convergence}. The update rule of Muon is characterized and extended via the linear minimization oracle (LMO) that accommodates arbitrary norms and layer‑wise structure in \citep{pethick2025training}. Furthermore, by selecting specific norms within the LMO framework, the resulting algorithm, Scion, exhibits the transferability of hyperparameters across model sizes. Gluon, proposed in \citep{riabinin2025gluon}, adopts the layer-wise $(L^0, L^1)$-smooth assumption \citep{zhang2019gradient} to capture the layer-wise geometry of neural networks, providing convergence guarantees with strong practical predictive power. Muon has also inspired many subsequent developments, including variance‑reduced methods \citep{sfyraki2025lions,chang2025convergence,huang2025limuonlightfastmuon,qian2025muon,kovalev2025non} and compressed methods for distributed learning \citep{gruntkowska2025error,zhang2025provable,qian2026communication,mishra2026signmuon,su2026muonq}.

	The LMO framework has the update of the form $X^{k+1} = X^k + \gamma_k D^k$. The search direction $D^k$ is the solution to the LMO: 
	\begin{equation}\label{eq:dkLMO}
		D^k \eqdef \arg\min \{  \langle M^k, D\rangle : \|D \|\leq 1  \}, 
	\end{equation}
	where $M^k$ is the momentum term, $\langle \cdot, \cdot \rangle$ refers to an inner product (trace inner product in matrix spaces), and $\|\cdot\|$ refers to an arbitrary norm. The idealized Muon has the above update with the spectral norm $\| \cdot\|_2$. 
	
	While most convergence analyses of Muon‑type methods consider the ideal update in the LMO (i.e., the exact solution to problem (\ref{eq:dkLMO})), the inexact Muon update has been studied exclusively by \citet{shulgin2025beyond}, which reveals a fundamental coupling between this inexactness and the optimal step size and momentum. Low‑rank Muon, proposed in \citep{he2025low}, can also be regarded as an inexact version of Muon, but it considers a different regime from ours: the search direction is the exact matrix sign of $M_Q^k$, where $M_Q^k$ is a low‑rank approximation of $M^k$. The inexact polar decomposition for Muon with Nesterov momentum was recently considered in \citep{choudhury2026muon}, but they only address the non‑degenerate case (where $\bar \gamma$ is required to be strictly less than $1$).
	
	To characterize the LMO approximation, \citet{shulgin2025beyond} propose the following assumption, under which convergence is guaranteed provided that the upper bound of ${ \delta_k }$ is strictly less than one.
	
	\begin{assumption}\label{as:dknorm-in}
		Let $D^k$ be the exact solution to (\ref{eq:dkLMO}), The inexact solution ${\hat D}^k$ is assumed to satisfy an additive error bound for some $\delta_k \geq 0$: 
		\begin{equation}\label{eq:dkdeltak}
			\| {\hat D}^k - D^k\| \leq \delta_k. 
		\end{equation}
	\end{assumption}

	\subsection{Three limitations of Assumption \ref{as:dknorm-in}}
	We consider problem (\ref{eq:dkLMO}) with the spectral norm $\|\cdot\|_2$. For $M^k \in \R^{m\times n}$ with rank $r\leq m \leq n$ ( without loss of generality, we assume $m\leq n$ for simplicity), let $M^k = U\Sigma V^\top$ be the SVD decomposition, where $U \in \R^{m\times m}$, $\Sigma \in \R^{m\times n}$, and $V \in \R^{n\times n}$. $\Sigma$ is a diagonal matrix with diagonal entries $\Sigma_{ii} = \sigma_i$ for $1\leq i \leq r$, where $\{ \sigma_i  \}_{i=1}^r$ are the singular values with $\sigma_1\geq \cdots \geq \sigma_r>0$, and $\Sigma_{ii} = 0$ for $i>r$. There are several limitations of the Assumption \ref{as:dknorm-in}. 
	
	(1) First, the exact solution to (\ref{eq:dkLMO}) is used in Assumption \ref{as:dknorm-in}, but the solution is not unique when $r<m$. Let $U = [U_r, U_\perp]$ and $V = [V_r, V_\perp]$, where $U_r \in \R^{m\times r}$ and $V_r \in \R^{n\times r}$. Then the solution of problem (\ref{eq:dkLMO}) is $D^k = - U_r V_r^\top - U_\perp Y V_\perp^\top$, for any $Y \in \R^{(m-r) \times (n-r)}$ with $\|Y\|_2 \leq 1$. The $D^k$ in the idealized Muon \citep{jordan6muon} is actually slightly different, which is the solution of the following problem:
	\begin{equation}\label{eq:dk-idealMuon}
		D^k \eqdef - \arg\min \{ \|D - M^k\|_{\rm F} : D^\top D = I \ {\rm or} \ DD^\top = I \}, 
	\end{equation}
	where $\|\cdot\|_{\rm F}$ is the Frobenius norm. The solution $D^k = -U_r V_r^\top - U_\perp Z V_\perp^\top$, for any $Z \in \R^{(m-r) \times (n-r)}$ with $ZZ^\top = I$. The iterative algorithms for solving (\ref{eq:dkLMO}), for example Newton-Schulz \citep{higham2008functions} and Polar Express \citep{amsel2025polar}, actually converge to the solution $D^k = -U_rV_r^\top$. 
	
	(2) To solve (\ref{eq:dkLMO}), $M^k$ is actually rescaled by $\|M^k\|_{\rm F} + \epsilon_{ns}$ to ensure that all its singular values lie in the interval $[0,1]$, and also to enhance numerical stability. Then the $r$ positive singular values for the rescaled matrix become 
	\begin{equation}\label{eq:sv-rmk}
		\frac{\sigma_i}{\sqrt{\sum_{i=1}^r \sigma_i^2} + \epsilon_{ns}}, 
	\end{equation}
	for $1\leq i\leq r$. Iterative algorithms for solving (\ref{eq:dkLMO})—such as Newton–Schulz \citep{higham2008functions}, QWHD iteration \citep{nakatsukasa2010optimizing,nakatsukasa2013stable}, Zolotarev polar decomposition \citep{nakatsukasa2016computing}, and Polar Express \citep{amsel2025polar}—typically require a positive lower bound on the singular values of the rescaled matrix $M^k$ to ensure convergence. However, when $\sigma_r \ll \sigma_1$, the smallest positive singular value in (\ref{eq:sv-rmk}) approaches zero, which implies that $\delta_k$ in Assumption \ref{as:dknorm-in} approaches $1$. Yet, the convergence guarantee in \citet{shulgin2025beyond} requires the upper bound on $\delta_k$ in Assumption \ref{as:dknorm-in} to be strictly less than $1$. 
	
	(3) When the algorithm converges, $\|\nabla f(x^k)\|_{\star}$ and $\|M^k - \nabla f(x^k)\|_\star$ tend to zero, where $\|\cdot\|_\star$ is the dual norm of $\|\cdot\|$. Consequently, $\|M^k\|_\star$ also tends to zero. Thus, when $\sigma_1 \ll \epsilon_{ns}$, the singular values in (\ref{eq:sv-rmk}) approach zero as well.

	When a matrix is downcast to a lower arithmetic precision, the small singular values tend to increase \citep{boutsikas2024small}, and then from their results, it is claimed in \citep{amsel2025polar} that for matrices stored in floating point arithmetic, the singular values are usually larger than machine precision. However, the machine precision of double precision, which is about $2.2\times 10^{-16}$, is extremely small. Furthermore, we construct two examples in Appendix \ref{sec:twoexa} to show that when a matrix in double precision is demoted to lower precision, the smallest singular value can be much smaller than the corresponding machine precision. Specifically, the smallest singular value for bfloat16 \citep{kalamkar2019study} can be on the order of $10^{-12}$ in some cases. This raises the following key problem: 
	\vskip 2mm
	{\em Can we still get the convergence of Muon-type methods with the inexact LMO in the degenerate case where $\delta_k$ in Assumption \ref{as:dknorm-in} approaches $1$? }
	\vskip 2mm
	In this work, we address the problem in the affirmative.

	\subsection{Contributions}
	
	Our key contributions can be summarized as follows:
	\begin{itemize}
		\item[(1)] We propose novel assumptions to tackle the degenerate scenario, and establish convergence guarantees for Gluon with the inexact LMO under the layer-wise $(L^0, L^1)$-smooth assumption. Our results characterize the relationship between the precision and the lower bound required by the iterative solver for the LMO. 
		
		\item[(2)] We establish convergence for Gluon with the inexact LMO under the layer-wise $(L^0, L^1)$-smooth assumption in the deterministic setting. Furthermore, we obtain the linear convergence rate under the additional layer-wise P\L{} condition, which improves upon the sublinear convergence rate of Gluon in the case where $L_i^1>0$ \citep{riabinin2025gluon}. 
		
		\item[(3)] We study Gluon with the inexact LMO and weight decay in the star-convex setting, and establish convergence under the layer-wise $(L^0, L^1)$-smooth assumption for both deterministic and stochastic cases. The resulting convergence rates improve upon those in the general convex setting. When the LMO is solved exactly, the convergence of Muon with weight decay in the star‑convex setting was established in \citep{kovalev2025understanding}; however, our convergence result under the layer‑wise $(L^0, L^1)$-smooth assumption is novel. 
	\end{itemize}
	
	\section{Solution to the limitations}
	
	In this section, we discuss how to address the three limitations of Assumption \ref{as:dknorm-in}. First, we deal with the first two limitations by proposing the following conditions: $\|{\hat D}^k \| \leq 1+\delta_k$ and $\langle M^k, {\hat D}^k \rangle \leq -(1-\delta_k) \|M^k\|_\star$\footnote{The objective function $\langle M^k, \cdot \rangle$ in the LMO was also recently used in \citep{choudhury2026muon} to characterize the inexact polar decomposition.}. The minimum of $\langle M^k, D \rangle$ over $\|D\|\leq 1$ is actually $-\|M^k\|_\star$. It is easy to see that Assumption \ref{as:dknorm-in} is a sufficient condition for our proposed conditions. By using the objective function $\langle M^k, \cdot \rangle$ in (\ref{eq:dkLMO}), we can deal with the case where the solution to (\ref{eq:dkLMO}) is not unique. Next we consider the case in the second limitation where $\sigma_1$ is not small, but $\sigma_r/\sigma_1$ approaches zero. When the norm in (\ref{eq:dkLMO}) is the spectral norm, for $\sigma_1 \geq \epsilon_{ns}$, assume ${\tilde r} \leq r$ is an integer such that $\sigma_i \geq \frac{c\sigma_1}{r}$ for $i\leq \tilde r$ and $\sigma_i < \frac{c\sigma_1}{r}$ for the rest $i$, for some constant $0<c\leq 1$. Then for $i\leq {\tilde r}$, (\ref{eq:sv-rmk}) is bounded below by $\frac{c}{r(\sqrt{r}+1)}$. For the above lower-bound, we can use some iterative algorithm to get ${\hat D}^k = -U{\hat \Lambda}V^\top$ such that $|{\hat \Lambda}_{ii} - 1| \leq {\tilde \delta}_k$ for $i \leq {\tilde r}$, and $0\leq {\hat \Lambda}_{ii} \leq 1+ (c+{\tilde \delta}_k)/(1+c)$ for $i>{\tilde r}$. Then, from Proposition \ref{pro:computedelk}, our proposed conditions are satisfied with $\delta_k = (c+{\tilde \delta}_k)/(1+c)$. To further handle the case in the last limitation where all the singular values of $M^k$ approach zero, we propose the following assumption. 
	
	\begin{assumption}\label{as:inexactLMO}
		The inexact solution ${\hat D}^k$ to (\ref{eq:dkLMO}) and $M^k$ are assumed to satisfy the following conditions: i) $\| {\hat D}^k\| \leq 1 + \delta_k$; ii) either $\|M^k\|_\star \leq \epsilon_1$ or 
		\begin{equation}\label{eq:hatdk-obj}
			\langle M^k, {\hat D}^k \rangle \leq -(1-\delta_k) \|M^k\|_\star, 
		\end{equation}
		where $\epsilon_1\geq 0$ and $0\leq \delta_k<1$. 
	\end{assumption}
	
	Assumption \ref{as:inexactLMO} enables us to treat separately the case where all singular values of $M^k$ approach zero and the case where $\|M^k\|_\star > \epsilon_1$ (requiring condition (\ref{eq:hatdk-obj}) only in this regime). When the norm in (\ref{eq:dkLMO}) is the spectral norm, for $\|M^k\|_\star > \epsilon_1$, we know $\sigma_1 \geq \frac{\epsilon_1}{r}$. Then from Proposition \ref{pro:computedelk}, for any $0<c\leq 1$, the rescaled singular value (\ref{eq:sv-rmk}) is lower below by $l = \frac{c\epsilon_1}{r\sqrt{r} (\epsilon_1 + \sqrt{r}\epsilon_{ns})}$ for $i\leq \tilde r$ ($\tilde r$ is defined in Proposition \ref{pro:computedelk} (i)), and condition (\ref{eq:hatdk-obj}) can be satisfied with $\delta_k = (c+{\tilde \delta}_k)/(1+c)$, where ${\tilde \delta}_k$ is the precision of the iterative algorithm for solving (\ref{eq:dkLMO}).

	Note that when $\|M^k\|_\star \leq \epsilon_1$, $\langle M^k, {\hat D}^k \rangle \leq 0$ can be easily satisfied. Hence, we also propose the following variant of Assumption \ref{as:inexactLMO}.

	\begin{assumption}\label{as:inexactLMO-2}
		The inexact solution ${\hat D}^k$ to (\ref{eq:dkLMO}) and $M^k$ are assumed to satisfy the following conditions: i) $\| {\hat D}^k\| \leq 1 + \delta_k$; ii) either $\|M^k\|_\star \leq \epsilon_1$ and $\langle M^k, {\hat D}^k \rangle \leq 0$ or 
		\begin{equation}\label{eq:hatdk-obj-2}
			\langle M^k, {\hat D}^k \rangle \leq -(1-\delta_k) \|M^k\|_\star, 
		\end{equation}
		where $\epsilon_1\geq 0$ and $0\leq \delta_k<1$. 
	\end{assumption}
	
	\paragraph{The relations between the lower bound $l$, ${\tilde \delta}^k$, and the inner iteration number $T$} We use the iterative algorithm Polar Express as an example. From Theorem 4.3 in \citep{amsel2025polar}, for the odd polynomial with degree $2q+1$, to reach the precision $|{\hat \Lambda}_{ii} - 1| \leq {\tilde \delta}_k$ for $i \leq {\tilde r}$, it is sufficient to choose $T$ such that 
	$$
	(1-l^2)^{(q+1)^T} \leq {\tilde \delta}_k. 
	$$
	Then we can choose $T \geq \frac{1}{\ln(q+1)} \left(  2\ln \frac{1}{l} + \ln \ln \frac{1}{{\tilde \delta}_k}  \right)$ to satisfy the above condition. Next we analyze the polynomial behavior on $[0, l]$. From Proposition \ref{pro:ptv}, we know $p_t(x) \in [0, l_{t+1}]$ ($p_t(\cdot)$ is the used polynomial at the $t$-th iteration) on $[0, l_t]$ for $t=1, ..., T$. Hence we have $0 \leq {\hat \Lambda}_{ii} \leq l_{T+1} \leq 1$ for $i > {\tilde r}$.

	\section{The convergence of Gluon under inexact LMO}

	In this work, we consider the layer-wise case, where $X = [X_1, ..., X_p]$ with $X_i \in {\cal S}_i$, and denote the used norm in ${\cal S}_i$ as $\|\cdot\|_{(i)}$ with associated dual norm $\|\cdot\|_{(i)\star}$. We extend Assumptions \ref{as:inexactLMO} and \ref{as:inexactLMO-2} to the layer-wise case as follows. 
	\begin{assumption}\label{as:inexactLMO-layer}
		The inexact solution ${\hat D}_i^k$ to the following LMO:
		\begin{equation}\label{eq:dkLMO-layer}
			D_i^k \eqdef \arg\min \{  \langle M_i^k, D_i \rangle : \|D_i\|_{(i)}\leq 1  \}, 
		\end{equation}
		and $M_i^k$ are assumed to satisfy the following conditions: i) $\| {\hat D}_i^k\|_{(i)} \leq 1 + \delta_{k, i}$; ii) either $\|M_i^k\|_{(i)\star} \leq \epsilon_{1,i}$ or 
		\begin{equation}\label{eq:hatdk-obj-layer}
			\langle M_i^k, {\hat D}_i^k \rangle \leq -(1-\delta_{k, i}) \|M_i^k\|_{(i)\star}, 
		\end{equation}
		where $\epsilon_{1, i}\geq 0$ and $0\leq \delta_{k, i}<1$. 
	\end{assumption}
	
	\begin{assumption}\label{as:inexactLMO-layer-2}
		The inexact solution ${\hat D}_i^k$ to (\ref{eq:dkLMO-layer}) and $M_i^k$ are assumed to satisfy the following conditions: i) $\| {\hat D}_i^k\|_{(i)} \leq 1 + \delta_{k, i}$; ii) either $\|M_i^k\|_{(i)\star} \leq \epsilon_{1,i}$ and $\langle M_i^k, {\hat D}_i^k \rangle \leq 0$ or 
		\begin{equation}\label{eq:hatdk-obj-layer-2}
			\langle M_i^k, {\hat D}_i^k \rangle \leq -(1-\delta_{k, i}) \|M_i^k\|_{(i)\star}, 
		\end{equation}
		where $\epsilon_{1, i}\geq 0$ and $0\leq \delta_{k, i}<1$. 
	\end{assumption}
	
	From the above assumptions, we define a sequence set ${\cal I}^k$ for $k\geq 0$ as 
	$$
	\I^k \eqdef \{  i \in \{1, ..., p\} |  \langle M_i^k, {\hat D}_i^k \rangle > -(1-\delta_{k, i}) \|M_i^k\|_{(i)\star}  \}. 
	$$
	Then under Assumptions \ref{as:inexactLMO-layer} and \ref{as:inexactLMO-layer-2}, we know that for any $i \in \I^k$, we have $\|M_i^k\|_{(i)\star} \leq \epsilon_{1,i}$. 
	
	We adopt the following layer-wise $(L^0, L^1)$-smoothness assumption used in Gluon \citep{riabinin2025gluon} to enable us to explore the layer-wise geometry of neural networks in the training. 
	
	\begin{assumption}\label{as:L0L1smooth}
		The function $f: \mathcal{S} \mapsto \mathbb{R}$ is layer-wise $(L^0, L^1)$-smooth with constants $L^0 := (L^0_1, \ldots, L^0_p) \in \mathbb{R}_+^p$ and $L^1 := (L^1_1, \ldots, L^1_p) \in \mathbb{R}_+^p$. That is, the inequality
		\begin{eqnarray*}
			\quad  \|\nabla_i f(X) - \nabla_i f(Y)\|_{(i)\star} \leq \left(L^0_i + L^1_i \|\nabla_i f(X)\|_{(i)\star}\right) \|X_i - Y_i\|_{(i)}
			\label{eq:layer_smooth}
		\end{eqnarray*}
		holds for all $i = 1, \ldots, p$ and all $X = [X_1, \ldots, X_p] \in \mathcal{S}$, $Y = [Y_1, \ldots, Y_p] \in \mathcal{S}$.
	\end{assumption}
	
	\begin{assumption}\label{as:boundedvariance}
		The stochastic gradient estimator $\nabla f_\xi : {\cal S} \to {\cal S}$ is unbiased and has bounded variance. That is, $\E_{\xi \sim {\cal D}} \left[  \nabla f_\xi (X)  \right] = \nabla f(X)$ for all $X \in {\cal S}$ and there exists $\sigma \geq 0$ such that
		$$
		\E_{\xi \sim {\cal D}} \left[  \|\nabla_i f_\xi (X) - \nabla_i f(X) \|_{\rm F}^2 \right] \leq \sigma^2, 
		$$
		for any $X \in {\cal S}$ and $ i = 1, ..., p$. 
	\end{assumption}

	\begin{assumption}\label{as:rho}
		$\| X\|_{(i)\star} \leq \rho\|X\|_{\rm F}$ for any $X$ and $i \in \{1, ..., p\}$. 
	\end{assumption}

	\begin{algorithm}
		\caption{Gluon with inexact LMO}\label{alg:gluon-inexact}
		\begin{algorithmic}[1]
			\State \textbf{Input:} Initial model parameters $X^0 = [X_1^0, \dots, X_p^0] \in \mathcal{S}$, momentum $M^0 = [M_1^0, \dots, M_p^0] \in \mathcal{S}$, momentum decay factors $\beta \in [0, 1)$ for all iterations $k \geq 0$
			\For{$k = 0, 1, 2, \dots, K - 1$}
			\State Sample $\xi^k \sim \mathcal{D}$
			\For{$i = 1, 2, \dots, p$}
			\State Compute stochastic gradient $\nabla_i f_{\xi^k}(X^k)$  for layer $i$
			\State Update momentum $M_i^k =  \beta M_i^{k-1} + (1-\beta) \nabla_i f_{\xi^k} (X^k) $ for layer $i$
			\State Choose adaptive stepsize/radius $t_i \eta > 0$  for layer $i$
			\State Compute inexact LMO with $M_i^k$ for layer $i$: $${\hat D}_i^k \approx \argmin_{\|D_i\|_{(i)}\leq 1} \langle M_i^k, D_i \rangle $$
			\State Update parameters for layer $i$: $X_i^{k+1} = X_i^k + t_i \eta {\hat D}_i^k$
			\EndFor
			\State Update full parameter vector $X^{k+1} = [X_1^{k+1}, \dots, X_p^{k+1}]$
			\EndFor
		\end{algorithmic}
	\end{algorithm}

	From above assumptions, we can obtain the following convergence results for Gluon with inexact LMO (Algorithm \ref{alg:gluon-inexact}). 
	
	\begin{theorem}\label{th:Inexact-gluon}
		Let Assumptions \ref{as:L0L1smooth}, \ref{as:boundedvariance}, and Assumption \ref{as:rho} hold. Let Assumption \ref{as:inexactLMO-layer} or \ref{as:inexactLMO-layer-2} hold with $\delta_{k, i} = \delta <1$. Define the constant $\phi=2$ if Assumption \ref{as:inexactLMO-layer} holds, and $\phi = 1-\delta$ if Assumption \ref{as:inexactLMO-layer-2} holds. Let $X^0, ..., X^{K-1}$ be the iterates of Algorithm \ref{alg:gluon-inexact}, $\alpha \eqdef 1-\beta$, and $M_i^{-1} = \nabla_i f_{\xi^0} (X^0)$. 
		
		1. If $L_i^1 =0$, then
		\begin{eqnarray*}
			&& \min_{k=0, ..., K-1} \sum_{i=1}^p t_i \E \left[  \|\nabla_i f(X^k) \|_{(i)\star} \right] \\ 
			&\leq&  \frac{1}{1-\delta} \left[ \frac{\Delta^0}{\eta K} + \frac{2\sum_{i=1}^p t_i \rho \sigma}{\alpha K} + 2\sqrt{\alpha} \sum_{i=1}^p t_i \rho \sigma + \frac{2(1+\delta)\sum_{i=1}^p L_i^0t_i^2 \eta}{\alpha} \right. \\ 
			&& \left. \hskip 14mm  + \frac{(1+\delta)^2\sum_{i=1}^p L_i^0 t_i^2 \eta}{2}  + \sum_{i=1}^p \phi t_i\epsilon_{1, i}  \right], 
		\end{eqnarray*}
		where $\Delta^0 \eqdef f(X^0) - \inf_{X \in {\cal S}} f(X)$. 
		
		2. If $L_i^1 \neq 0$, we let $\frac{\eta}{\alpha} \leq \min_{i}  \frac{1}{6(1+\delta)L_i^1 t_i}$. Then $\left(  \frac{2(1+\delta)}{\alpha} + \frac{(1+\delta)^2}{2}  \right) L_i^1 t_i \eta \leq \frac{1}{2}$ for all $i$, and 
		\begin{eqnarray*}
			&& \min_{k=0, ..., K-1} \sum_{i=1}^p t_i \E \left[  \|\nabla_i f(X^k) \|_{(i)\star} \right] \\ 
			&\leq&   \frac{1}{1-\delta} \left[ \frac{2\Delta^0}{\eta K} + \frac{4\sum_{i=1}^p t_i \rho \sigma}{\alpha K} + 4\sqrt{\alpha} \sum_{i=1}^p t_i \rho \sigma + \frac{4(1+\delta)\sum_{i=1}^p L_i^0t_i^2 \eta}{\alpha} \right. \\ 
			&& \left. \hskip 14mm  + {(1+\delta)^2\sum_{i=1}^p L_i^0 t_i^2 \eta}  + \sum_{i=1}^p 2 \phi t_i\epsilon_{1, i}  \right]. 
		\end{eqnarray*}
	\end{theorem}

	\begin{corollary}\label{co:inexact-gluon}
		Under the premise of Theorem \ref{th:Inexact-gluon}, by choosing $\eta = \frac{1}{K^{\frac{3}{4}} (1+\delta)^{\frac{1}{4}}}$ and $\alpha = \frac{\sqrt{1+\delta}}{\sqrt{K}}$, we have the following results. 
		1. If $L_i^1 =0$, then
		\begin{eqnarray*}
			\min_{k=0, ..., K-1} \sum_{i=1}^p t_i \E \left[  \|\nabla_i f(X^k) \|_{(i)\star} \right] \leq {\cal O} \left(  \frac{(1+\delta)^{1/4}}{(1-\delta) K^{1/4}}  \right)   +  \sum_{i=1}^p \frac{\phi t_i \epsilon_{1, i}}{1-\delta}. 
		\end{eqnarray*}
		2. If $L_i^1 \neq 0$ and $K \geq (1+\delta) \max_i (6L_i^1 t_i)^4$. Then $\frac{\eta}{\alpha} \leq \min_{i}  \frac{1}{6(1+\delta)L_i^1 t_i}$, and 
		\begin{eqnarray*}
			\min_{k=0, ..., K-1} \sum_{i=1}^p t_i \E \left[  \|\nabla_i f(X^k) \|_{(i)\star} \right] \leq {\cal O} \left(  \frac{(1+\delta)^{1/4}}{(1-\delta) K^{1/4}}  \right)   +  \sum_{i=1}^p \frac{2\phi t_i \epsilon_{1, i}}{1-\delta}. 
		\end{eqnarray*}
	\end{corollary}
	
	From Corollary \ref{co:inexact-gluon}, to reach the precision 	$ \min_{k=0, ..., K-1} \sum_{i=1}^p t_i \E \left[  \|\nabla_i f(X^k) \|_{(i)\star} \right]  \leq \epsilon$, it is sufficient to choose the parameters such that $\sum_{i=1}^p \frac{2\phi t_i \epsilon_{1, i}}{1-\delta} \leq {\tilde c}\epsilon$ and ${\cal O} \left(  \frac{(1+\delta)^{1/4}}{(1-\delta) K^{1/4}}  \right) \leq (1-{\tilde c})\epsilon$ for any $0<{\tilde c}<1$.

	\section{Inexact Gluon in the deterministic case}

	In this section, we consider the Gluon with the inexact LMO in the deterministic case (Algorithm \ref{alg:gluon-inexact-deter}), that is, we calculate the full gradient at each step. Similar to the stochastic case, we propose the following assumptions related to the LMO approximation. 
	
	\begin{assumption}\label{as:inexactLMO-layer-deter}
		The inexact solution ${\hat D}_i^k$ to the following LMO: 
		\begin{equation}\label{eq:dkLMO-deter}
			D_i^k \eqdef \arg\min \{  \langle \nabla_i f(X^k), D_i \rangle : \|D_i\|_{(i)}\leq 1  \}, 
		\end{equation}
		and $\nabla_i f(X^k)$ are assumed to satisfy the following conditions: i) $\| {\hat D}_i^k\|_{(i)} \leq 1 + \delta_{k, i}$; ii) either $\|\nabla_i f(X^k)\|_{(i)\star} \leq \epsilon_{1,i}$ or 
		\begin{equation}\label{eq:hatdk-obj-layer-deter}
			\langle \nabla_i f(X^k), {\hat D}_i^k \rangle \leq -(1-\delta_{k, i}) \|\nabla_i f(X^k)\|_{(i)\star}, 
		\end{equation}
		where $\epsilon_{1, i}\geq 0$ and $0\leq \delta_{k, i}<1$. 
	\end{assumption}
	
	\begin{assumption}\label{as:inexactLMO-layer-2-deter}
		The inexact solution ${\hat D}_i^k$ to (\ref{eq:dkLMO-deter}) and $\nabla_i f(X^k)$ are assumed to satisfy the following conditions: i) $\| {\hat D}_i^k\|_{(i)} \leq 1 + \delta_{k, i}$; ii) either $\|\nabla_i f(X^k)\|_{(i)\star} \leq \epsilon_{1,i}$ and $\langle \nabla_i f(X^k), {\hat D}_i^k \rangle \leq 0$ or 
		\begin{equation}\label{eq:hatdk-obj-layer-2-deter}
			\langle \nabla_i f(X^k), {\hat D}_i^k \rangle \leq -(1-\delta_{k, i}) \|\nabla_i f(X^k)\|_{(i)\star}, 
		\end{equation}
		where $\epsilon_{1, i}\geq 0$ and $0\leq \delta_{k, i}<1$. 
	\end{assumption}
	
	Same as the stochastic case, from the above assumptions, we define a sequence set ${\cal I}_D^k$ for $k\geq 0$ as 
	$$
	\I_D^k \eqdef \{  i \in \{1, ..., p\} |  \langle  \nabla_i f(X^k), {\hat D}_i^k \rangle > -(1-\delta_{k, i}) \| \nabla_i f(X^k)\|_{(i)\star}  \}. 
	$$
	Then under Assumptions \ref{as:inexactLMO-layer-deter} and \ref{as:inexactLMO-layer-2-deter}, we know that for any $i \in \I_D^k$, we have $\| \nabla_i f(X^k)\|_{(i)\star} \leq \epsilon_{1,i}$. We can get the following convergence results for Algorithm \ref{alg:gluon-inexact-deter}.

	\begin{algorithm}
		\caption{Gluon with inexact LMO in the deterministic case}\label{alg:gluon-inexact-deter}
		\begin{algorithmic}[1]
			\State \textbf{Input:} Initial model parameters $X^0 = [X_1^0, \dots, X_p^0] \in \mathcal{S}$ 
			\For{$k = 0, 1, 2, \dots, K - 1$}
			\For{$i = 1, 2, \dots, p$}
			\State Compute gradient $\nabla_i f(X^k)$  for layer $i$
			\State Choose adaptive stepsize/radius $t_i^k > 0$  for layer $i$
			\State Compute inexact LMO with $\nabla_i f(X^k)$ for layer $i$: $${\hat D}_i^k \approx \argmin_{\|D_i\|_{(i)}\leq 1} \langle \nabla_i f(X^k), D_i \rangle $$
			\State Update parameters for layer $i$: $X_i^{k+1} = X_i^k + t_i^k {\hat D}_i^k$
			\EndFor
			\State Update full parameter vector $X^{k+1} = [X_1^{k+1}, \dots, X_p^{k+1}]$
			\EndFor
		\end{algorithmic}
	\end{algorithm}

	\begin{theorem}\label{th:gluon-deter}
		Let Assumption \ref{as:L0L1smooth} hold and fix $\epsilon>0$. Let Assumption \ref{as:inexactLMO-layer-deter} or \ref{as:inexactLMO-layer-2-deter} hold with $\delta_{k, i} = \delta <1$. Define the constant $\phi=2$ if Assumption \ref{as:inexactLMO-layer-deter} holds, and $\phi = 1-\delta$ if Assumption \ref{as:inexactLMO-layer-2-deter} holds. Let $X^0, ..., X^{K-1}$ be the iterates of Algorithm \ref{alg:gluon-inexact-deter} with stepsizes $t_i^k = \frac{(1-\delta) \|\nabla_i f(X^k)\|_{(i)\star}}{(1+\delta)^2 (L_i^0 + L_i^1 \|\nabla_i f(X^k)\|_{(i)\star})}$. Then for any $0<\tilde c<1$, 
		
		\noindent 1. In order to reach the precision 
		$$
		\min_{k=0, ..., K-1} \sum_{i=1}^p \|\nabla_i f(X^k) \|_{(i)\star} \leq \epsilon, 
		$$
		it is sufficient to choose 
		$$
		K = \left\lceil  \frac{2(1+\delta)^2 \sum_{i=1}^p L_i^0 \Delta^0}{(1-\delta)^2(1-\tilde c)\epsilon^2}  +  \frac{2(1+\delta)^2 L_{\max}^1 \Delta^0}{(1-\delta)^2 (1-\tilde c)\epsilon}  \right\rceil, 
		$$
		and $\epsilon_{1, i}$ such that $\sum_{i=1}^p \frac{\phi \epsilon_{1, i}^2}{(1-\delta) (L_i^0 + L_i^1 \epsilon_{1, i})} \leq \frac{{\tilde c}\epsilon^2}{2(\sum_{i=1}^p L_i^0 + L_{\max}^1 \epsilon)}$; 
		
		\noindent 2. In order to reach the precision 
		$$
		\min_{k=0, ..., K-1} \sum_{i=1}^p \frac{1/L_i^1}{\frac{1}{p} \sum_{i=1}^p1/L_i^1}\|\nabla_i f(X^k) \|_{(i)\star} \leq \epsilon, 
		$$
		it is sufficient to choose 
		$$
		K = \left\lceil   \frac{2(1+\delta)^2\Delta^0 (\sum_{i=1}^p \frac{L_i^0}{(L_i^1)^2})}{(1-\delta)^2(1-\tilde c)\epsilon^2 \left(  \frac{1}{p} \sum_{j=1}^p \frac{1}{L_j^1} \right)^2}   +  \frac{2(1+\delta)^2\Delta^0}{(1-\delta)^2(1-\tilde c) \epsilon \left(  \frac{1}{p} \sum_{j=1}^p \frac{1}{L_j^1} \right)} \right\rceil, 
		$$
		and $\epsilon_{1, i}$ such that $$\sum_{i=1}^p \frac{\phi \epsilon_{1, i}^2}{(1-\delta) (L_i^0 + L_i^1 \epsilon_{1, i})} \leq \frac{{\tilde c} \left(\frac{\epsilon}{p} \sum_{j=1}^p \frac{1}{L_j^1}\right)^2}{2\left(  \sum_{i=1}^p \frac{L_i^0}{(L_i^1)^2 } + \frac{\epsilon}{p} \sum_{j=1}^p \frac{1}{L_j^1} \right)}.$$

	\end{theorem}

	\subsection{Convergence under the P\L{} condition}
	
	In this subsection, we consider Algorithm \ref{alg:gluon-inexact-deter} under the following layer-wise Polyak–\L{}ojasiewicz (P\L{}) condition introduced in \citep{riabinin2025gluon}. 
	
	\begin{assumption}[Layer-wise Polyak-\L ojasiewicz condition] \label{as:PL}
		The function $f: \mathcal{S} \mapsto \mathbb{R}$ satisfies the layer-wise Polyak-\L{}ojasiewicz (P\L{}) condition with a constant $\mu > 0$, i.e., for any $X \in \mathcal{S}$
		$$
		\sum_{i=1}^{p} \|\nabla_i f(X)\|_{(i)\star}^2 \geq 2\mu \left(f(X) - f^\star\right),
		$$
		where $f^\star \eqdef \inf_{X \in \mathcal{S}} f(X) > -\infty$.
	\end{assumption}

	With the additional layer-wise P\L{} condition, we can establish the following linear convergence rates under some restrictions on $\epsilon_{1, i}$. Our convergence rate improves upon the sublinear rate $\mathcal{O}(1/\epsilon)$ achieved by \citet{riabinin2025gluon} for the case where $L_i^1 > 0$. 
	
	\begin{theorem}\label{th:gluon-deter-pl}
		Let Assumption \ref{as:L0L1smooth} and Assumption \ref{as:PL} hold and fix $\epsilon>0$. Let Assumption \ref{as:inexactLMO-layer-deter} or \ref{as:inexactLMO-layer-2-deter} hold with $\delta_{k, i} = \delta <1$. Define the constant $\phi=2$ if Assumption \ref{as:inexactLMO-layer-deter} holds, and $\phi = 1-\delta$ if Assumption \ref{as:inexactLMO-layer-2-deter} holds. Let $X^0, ..., X^{K-1}$ be the iterates of Algorithm \ref{alg:gluon-inexact-deter} run with stepsizes $t_i^k = \frac{(1-\delta) \|\nabla_i f(X^k)\|_{(i)\star}}{(1+\delta)^2 (L_i^0 + L_i^1 \|\nabla_i f(X^k)\|_{(i)\star})}$. Then, 
		
		\noindent 1. If $p=1$, then to reach the precision $\min_{k=0, ..., K-1} f(X^k) - f^\star \leq \epsilon$, it is sufficient to choose $\epsilon_{1, 1} = \sqrt{2\mu \epsilon}$ and 
		$$
		K =  \left\lceil  \frac{(1+\delta)^2 (L_1^0 + \sqrt{2\mu \Delta^0}L_1^1)}{(1-\delta)^2 \mu} \ln \frac{\Delta^0}{\epsilon} \right\rceil, 
		$$
		
		\noindent 2. For any $0<\tilde c<1$, to reach the precision $f(X^K) - f^\star \leq \epsilon$, it is sufficient to choose 
		$$
		K = \left\lceil  \frac{(1+\delta)^2(L_{\max}^0 + L_{\max}^1 \sqrt{2\mu \max\{ \Delta^0, \epsilon\}})}{(1-\delta)^2\mu} \ln \frac{\Delta^0}{(1-\tilde c)\epsilon}   \right\rceil, 
		$$
		and $\epsilon_{1, i}$ such that $$\sum_{i=1}^p \frac{\phi (1-\delta)\epsilon_{1, i}^2}{(1+\delta)^2 (L_i^0 + L_i^1 \epsilon_{1, i})} \leq \frac{(1-\delta)^2\mu}{(1+\delta)^2(L_{\max}^0 + L_{\max}^1 \sqrt{2\mu \max\{ \Delta^0, \epsilon\}})} {\tilde c}\epsilon,$$ 
		where $L_{\max}^0 \eqdef \max_{i=1, ..., p} L_i^0$, $L_{\max}^1 \eqdef \max_{i=1, ..., p} L_i^1$, and $\Delta^0 \eqdef f(X^0) - f^\star$.
	\end{theorem}

	\section{Inexact Gluon with weight decay for the star-convex case}

	In this section, we consider the star-convex case where the objective function $f(\cdot)$ in problem (\ref{p:p}) is star-convex (defined below). 
	
	\begin{assumption}[Star-convexity]\label{as:starconvex}
		We assume that the objective function $f(\cdot)$ in problem (\ref{p:p}) is star-convex, that is, the following inequality holds:
		\begin{equation}
			f(\beta X^* + (1 - \beta)X) \leq \beta f(X^*) + (1-\beta)f(X), \nonumber 
		\end{equation}
		for all $X \in {\cal S}$ and $\beta \in [0, 1]$, where $X^* \in {\cal S}$ is a solution. 
	\end{assumption}
	
	As pointed out by \citet{kovalev2025understanding}, it is possible to obtain improved convergence guarantees in the star-convex case as long as the iterations are bounded, and weight decay is an effective technique for ensuring such boundedness. We therefore integrate weight decay into inexact Gluon, resulting in Algorithm \ref{alg:gluon-sc}.

	\begin{algorithm}[t]
		\caption{Inexact Gluon with weight decay}
		\label{alg:gluon-sc}
		\begin{algorithmic}[1]
			\State \textbf{Input:} $X^0, M^{-1} \in \mathcal{S}$
			\State \textbf{parameters:} $\eta > 0$, $\alpha \in (0,1)$, $K \in \{1,2,\ldots\}$, {weight decay $\beta \in (0,1)$}
			\For{$k = 0$ to $K-1$}
			\State Sample $\xi^k \sim \mathcal{D}$
			\For{$i = 1, 2, \dots, p$}
			\State Compute $M_i^{k}$ for layer $i$ with two options:
			\State   \textbf{Option 1:} $M_i^k = \nabla_i f(X^k)$
			\State   \textbf{Option 2:} $M_i^k = (1 - \alpha) M_i^{k-1} + \alpha \nabla_i f_{\xi^k}(X^k)$
			\State Choose adaptive stepsize/radius $t_i \eta > 0$  for layer $i$
			\State Compute inexact LMO with $M_i^k$ for layer $i$: $${\hat D_i}^k \approx \argmin_{\|D_i\|_{(i)}\leq 1} \langle M_i^k, D_i \rangle $$
			\State Update parameters for layer $i$: $X_i^{k+1} = (1-\beta) X_i^k +  t_i\eta {\hat D_i}^k$
			\EndFor
			\State Update full parameter vector $X^{k+1} = [X_1^{k+1}, \dots, X_p^{k+1}]$
			\EndFor
		\end{algorithmic}
	\end{algorithm}

	\subsection{Deterministic case}
	
	First, we consider the deterministic case, i.e., the option 1 that $M^k \equiv \nabla f(X^k)$ for all $k$ in Algorithm \ref{alg:gluon-sc}. We can obtain the following convergence results. 
	
	\begin{theorem}\label{th:gluon-inexact-sc-deter}
		Let Assumptions \ref{as:L0L1smooth} and \ref{as:starconvex} hold. Let Assumption \ref{as:inexactLMO-layer-deter} or \ref{as:inexactLMO-layer-2-deter} hold with $\delta_{k, i} = \delta_i <1$. Define the constant $\phi_i=2$ if Assumption \ref{as:inexactLMO-layer-deter} holds, and $\phi_i = 1-\delta_i$ if Assumption \ref{as:inexactLMO-layer-2-deter} holds for all $i$. Assume 
		$$
		\beta \max_{i\in \{1, ..., p\}} \left\{  \frac{\|X_i^0\|_{(i)}}{(1+\delta_i) t_i},  \frac{\|X_i^*\|_{(i)}}{(1-s)(1-\delta_i)t_i}  \right\} \leq \eta, 
		$$
		with $s \in [0, 1)$, and $ s(1-\delta_i) \geq 2 (1 + (1+\delta_i)^2)L_i^1 t_i \eta$ for all $i$. 
		Let $X^0, ..., X^{K-1}$ be the iterates of Algorithm \ref{alg:gluon-sc} with Option 1. Then we have 
		$$
		f(X^K) - f(X^*) \leq (1-\beta)^K (f(X^0) - f(X^*)) + \frac{1}{\beta}  \sum_{i=1}^p  2 (1 + (1+\delta_i)^2) L_i^0 t_i^2 \eta^2 + \sum_{i=1}^p \frac{\phi_i t_i \eta \epsilon_{1, i}}{\beta}. 
		$$
	\end{theorem}

	\begin{corollary}\label{co:inexact-gluon-sc-deter}
		Under the premise of Theorem \ref{th:gluon-inexact-sc-deter}, denote $$D_s =  \max_{i\in \{1, ..., p\}} \left\{  \frac{\|X_i^0\|_{(i)}}{(1+\delta_i) t_i},  \frac{\|X_i^*\|_{(i)}}{(1-s)(1-\delta_i)t_i}  \right\},$$ and set $\beta = \frac{\ln K}{K}$, $\eta = \beta D_s$, then we have the following results. 
		1. For the case where $L_i^1 = 0$ for all $i$, by choosing $s=0$, we can obtain 
		$$
		f(X^K) - f(X^*) \leq \frac{f(X^0) - f(X^*)}{K} + \frac{\sum_{i=1}^p  2 (1 + (1+\delta_i)^2) L_i^0 t_i^2 D_s^2 \ln K}{K}  +  \sum_{i=1}^p \phi_i t_i D_s \epsilon_{1, i}. 
		$$
		2. For the case where $L_i^1 \neq 0$ for all $i$, we choose $s = \frac{1}{20}$, and assume $\frac{\ln K}{K} \leq \min_{i} \left\{  \frac{1-\delta_i}{40(1 + (1+\delta_i)^2)L_i^1 D_s t_i}  \right\}$. Then we have 
		$$
		f(X^K) - f(X^*) \leq \frac{f(X^0) - f(X^*)}{K} + \frac{\sum_{i=1}^p  2 (1 + (1+\delta_i)^2) L_i^0 t_i^2 D_s^2 \ln K}{K}  +  \sum_{i=1}^p \phi_i t_i D_s \epsilon_{1, i}. 
		$$
		
	\end{corollary}
	
	Compared to the $\mathcal{O}(1/\sqrt{K})$ convergence rate in Theorem~\ref{th:gluon-deter} for the general non‑convex case, the star‑convex case achieves a faster $\mathcal{O}(\ln K/K)$ rate in Corollary~\ref{co:inexact-gluon-sc-deter}.

	\subsection{Stochastic case}
	
	In this subsection, we consider the inexact Gluon with weight decay in the stochastic case, i.e., the option 2 in Algorithm \ref{alg:gluon-sc}. We have the following convergence results.

	\begin{theorem}\label{th:gluon-inexact-sc}
		Let Assumptions \ref{as:L0L1smooth}, \ref{as:boundedvariance}, \ref{as:rho}, and \ref{as:starconvex} hold. Let Assumption \ref{as:inexactLMO-layer} or \ref{as:inexactLMO-layer-2} hold with $\delta_{k, i} = \delta_i <1$. Define the constant $\phi_i=2$ if Assumption \ref{as:inexactLMO-layer} holds, and $\phi_i = 1-\delta_i$ if Assumption \ref{as:inexactLMO-layer-2} holds for all $i$. Assume 
		$$
		\beta \max_{i\in \{1, ..., p\}} \left\{  \frac{\|X_i^0\|_{(i)}}{(1+\delta_i) t_i},  \frac{\|X_i^*\|_{(i)}}{(1-s)(1-\delta_i)t_i}  \right\} \leq \eta, 
		$$
		with $s \in [0, 1)$, $\alpha > \beta$, and $s(1-\delta_i) \geq 2t_i \eta L_i^1 (1 + (1+\delta_i)^2 + \frac{5}{\alpha - \beta})$ for all $i$. 
		Let $X^0, ..., X^{K-1}$ be the iterates of Algorithm \ref{alg:gluon-sc} with Option 2, and $M^{-1} = \nabla f_{\xi^0}(X^0)$. Then we have 
		\begin{eqnarray*}
			\E [f(X^K) - f(X^*)] 
			&\leq& (1-\beta)^K (f(X^0) - f(X^*)) + \frac{1}{\beta}  \sum_{i=1}^p  2(1 + (1+\delta_i)^2)L_i^0 t_i^2 \eta^2  \\ 
			&& +  \sum_{i=1}^p \frac{\phi_i t_i \eta \epsilon_{1, i}}{\beta}   +   \sum_{i=1}^p (2 + s(1-\delta_i))t_i \eta \left(  \frac{\rho\sigma}{\alpha}  +  \frac{\sqrt{\alpha} \rho \sigma}{\beta}  +   \frac{2(1+\delta_i)L_i^0 t_i \eta}{\alpha \beta}   \right). 
		\end{eqnarray*}
	\end{theorem}

	In the following corollary, without loss of generality, we assume $D_s \geq 1$ and $\frac{D_s \ln K}{K}  \leq 1$ which generally holds for large $K$ such that $\beta < \alpha \leq 1$. We omit the proof since it is similar to that of Corollary \ref{co:inexact-gluon-sc-deter}.
	
	\begin{corollary}\label{co:inexact-gluon-sc}
		Under the premise of Theorem \ref{th:gluon-inexact-sc}, denote $$D_s =  \max_{i\in \{1, ..., p\}} \left\{  \frac{\|X_i^0\|_{(i)}}{(1+\delta_i) t_i},  \frac{\|X_i^*\|_{(i)}}{(1-s)(1-\delta_i)t_i}  \right\},$$ and set $\beta = \frac{\ln K}{K}$, $\eta = \beta D_s$, and $\alpha = \left(  \frac{D_s \ln K}{K}  \right)^{2/3}$, then we have the following results. 
		
		\noindent 1. For the case where $L_i^1 = 0$ for all $i$, by choosing $s=0$, we can obtain 
		\begin{eqnarray*}
			\E [f(X^K) - f(X^*)] &\leq& \frac{f(X^0) - f(X^*)}{K} + \frac{\sum_{i=1}^p  2 (1 + (1+\delta_i)^2) L_i^0 t_i^2 D_s^2 \ln K}{K}  +  \sum_{i=1}^p \phi_i t_i D_s \epsilon_{1, i} \\ 
			&& + \frac{\sum_{i=1}^p 2t_i \rho \sigma (D_s^{1/3} + D_s^{4/3}) (\ln K)^{1/3}}{K^{1/3}} + \frac{\sum_{i=1}^p 4(1+\delta_i) L_i^0 t_i^2 D_s^{4/3}(\ln K)^{1/3}}{K^{1/3}}. 
		\end{eqnarray*}
		
		\noindent 2. For the case where $L_i^1 \neq 0$ for all $i$, we choose $s = \frac{1}{3}$, and assume 
		$$
		\frac{\ln K}{K} + \frac{\left(  \frac{\ln K}{K}  \right)^{1/3}}{1 - \left(  \frac{\ln K}{K}  \right)^{1/3}} \leq \min_{i} \left\{  \frac{(1-\delta_i)}{30D_st_iL_i^1}  \right\}. 
		$$
		
		Then we have 
		\begin{eqnarray*}
			\E [f(X^K) - f(X^*)] &\leq& \frac{f(X^0) - f(X^*)}{K} + \frac{\sum_{i=1}^p  2 (1 + (1+\delta_i)^2) L_i^0 t_i^2 D_s^2 \ln K}{K}  +  \sum_{i=1}^p \phi_i t_i D_s \epsilon_{1, i} \\ 
			&& + \frac{\sum_{i=1}^p (\frac{7}{3} - \frac{\delta_i}{3})t_i \rho \sigma (D_s^{1/3} + D_s^{4/3}) (\ln K)^{1/3}}{K^{1/3}} \\ 
			&&  + \frac{\sum_{i=1}^p  (\frac{14}{3} - \frac{2\delta_i}{3}) (1+\delta_i) L_i^0 t_i^2 D_s^{4/3}(\ln K)^{1/3}}{K^{1/3}}. 
		\end{eqnarray*}
		
	\end{corollary}
	
	Compared to the $\mathcal{O}(\frac{1}{K^{1/4}})$ convergence rate in Corollary~\ref{co:inexact-gluon} for the general non‑convex case, the star‑convex case achieves a faster $\mathcal{O}(\frac{(\ln K)^{1/3}}{K^{1/3}})$ rate in Corollary~\ref{co:inexact-gluon-sc}.

	\bibliographystyle{plainnat} 
	\bibliography{ref_inexactLMO.bib}

	\clearpage
	\appendix

	\part*{Appendix}

	\tableofcontents
	
	\clearpage

	\section{Two examples}\label{sec:twoexa}

	In this section, we give two examples to show that when a matrix is demoted to lower precision, the smallest positive singular value can be much smaller than the corresponding machine precision. The matrix $A \in \R^{m\times n}$ is constructed in double precision, i.e., float64. Then $A$ is demoted to single precision, half precision, and bfloat16, and represented as ${\rm single}(A)$, ${\rm half}(A)$, and ${\rm bfloat16(A)}$, respectively. To calculate the singular values of the demoted versions, we transform them back to double precision (represented as ${\rm double}({\rm single}(A))$, ${\rm double}({\rm half}(A))$, and ${\rm double}({\rm bfloat16}(A))$, respectively), and use the function {\em numpy.linalg.svd} in Numpy library. We set the singular values that are smaller than $10^{-12}$ to $0$.

	\paragraph{First example} Let $A(col\_per, t) \eqdef \frac{{\tilde A}(col\_per, t)}{\|{\tilde A}(col\_per, t)\|_F + \epsilon_{ns}} \in \R^{m\times n}$, where $\epsilon_{ns} = 10^{-7}$ and ${\tilde A}(col\_per, t)$ is defined as follows. ${\tilde A}(col\_per, t)$ has two types of columns: type-$1$ column and type-$2$ column. Each element of type-$1$ column is randomly sampled from the interval $[-1, 1]$ uniformly and independently. $t$ elements of type-$2$ column is randomly sampled from the interval $[-1, 1]$ uniformly and independently, and the $t$ indices are chosen randomly with equal probability. The rest $m-t$ elements of type-$2$ column are randomly sampled from the interval $[-10^{-7}, 10^7]$ uniformly and independently. $col\_per\%$ columns are type-$1$ column, and the indices are chosen randomly with equal probability. The rest columns are type-$2$ column. We choose $col\_per \in \{ 10, 20, 40, 80, 90, 99 \}$ and $t \in \{0, 1, 2\}$ in the experiments. Hence most elements of type-2 column are close to $0$. We list the largest singular value of ${\rm double}(A)$ and smallest positive singular values of ${\rm double}({\rm single}(A))$, ${\rm double}({\rm half}(A))$, and ${\rm double}({\rm bfloat16}(A))$ in Tables \ref{tb:ex1-1}, \ref{tb:ex1-2}, and \ref{tb:ex1-3}. Note that the machine precisions of single precision, half precision, and bfloat16 are $2^{-23} \approx 1.19\times 10^{-7}$, $2^{-10} \approx 9.77\times 10^{-4}$, and $2^{-7} \approx 7.81 \times 10^{-3}$, respectively. In many cases, the smallest positive singular value of the demoted version can be much smaller than the corresponding machine precision.

	\paragraph{Second example} Let $A(t, \Delta_t) \eqdef \frac{{\tilde A}(t, \Delta_t)}{\|{\tilde A}(t, \Delta_t)\|_F + \epsilon_{ns}} \in \R^{m\times n}$, where $\epsilon_{ns} = 10^{-7}$ and we construct ${\tilde A}(t, \Delta_t)$ such that the first two columns are very close. The first $m-1$ elements of the first column of ${\tilde A}(t, \Delta_t)$ are randomly sampled from the interval $[-1, 1]$ uniformly and independently, and the last element is set to $t$. The values of the second column is the same as that of the first column except that the last element is set to $t+\Delta_t$ instead. The rest elements of ${\tilde A}(t, \Delta_t)$ are randomly sampled from the interval $[-1, 1]$ uniformly and independently. We choose $t \in \{  1, 10^{-1}, 10^{-2}, \dots, 10^{-8}  \}$, and for each $t$, $\Delta_t \in \{ t, t/10, t/100 \}$. Since the first two columns are close, the condition number of $A(t, \Delta_t)$ could be large. We list the largest singular value of ${\rm double}(A)$ and smallest positive singular values of ${\rm double}({\rm single}(A))$, ${\rm double}({\rm half}(A))$, and ${\rm double}({\rm bfloat16}(A))$ in Table \ref{tb:ex2} where $m=2000$ and $n=100$. The smallest positive singular value of the demoted version can also be much smaller than the corresponding machine precision in many cases.

	\begin{table}[htbp]
		\centering
		\footnotesize 
		\setlength{\tabcolsep}{3pt} 
		\renewcommand{\arraystretch}{1.2} 
		\caption{Singular values of $A(col\_per, t)$ under different precisions for $t=0$. }
		\label{tb:ex1-1}
		\begin{tabular}{c c c c c c}
			\toprule
			\makecell{($m$, $n$, $col\_per$, $t$)} & 
			\makecell{Largest singular \\ value of \\ ${\rm double}(A)$} & 
			\makecell{Smallest positive \\ singular value of \\ ${\rm double}(A)$} & 
			\makecell{Smallest positive \\ singular value of \\ ${\rm double}({\rm single}(A))$} & 
			\makecell{Smallest positive \\ singular value of \\ ${\rm double}({\rm half}(A))$} & 
			\makecell{Smallest positive \\ singular value of \\ ${\rm double}({\rm bfloat16}(A))$} \\ 
			\midrule
			(2000, 100, 10, 0) & $3.37 \times 10^{-1}$ & $2.51 \times 10^{-8}$ & $2.51 \times 10^{-8}$ & $2.95 \times 10^{-1}$ & $2.51 \times 10^{-8}$ \\
			(2000, 100, 20, 0) & $2.46 \times 10^{-1}$ & $1.79 \times 10^{-8}$ & $1.79 \times 10^{-8}$ & $2.04 \times 10^{-1}$ & $1.79 \times 10^{-8}$ \\
			(2000, 100, 40, 0) & $1.80 \times 10^{-1}$ & $1.31 \times 10^{-8}$ & $1.31 \times 10^{-8}$ & $1.38 \times 10^{-1}$ & $1.31 \times 10^{-8}$ \\
			(2000, 100, 80, 0) & $1.33 \times 10^{-1}$ & $1.01 \times 10^{-8}$ & $1.01 \times 10^{-8}$ & $9.00 \times 10^{-2}$ & $1.01 \times 10^{-8}$ \\
			(2000, 100, 90, 0) & $1.27 \times 10^{-1}$ & $9.65 \times 10^{-9}$ & $9.65 \times 10^{-9}$ & $8.37 \times 10^{-2}$ & $9.65 \times 10^{-9}$ \\
			(2000, 100, 99, 0) & $1.22 \times 10^{-1}$ & $9.73 \times 10^{-9}$ & $9.73 \times 10^{-9}$ & $7.79 \times 10^{-2}$ & $9.73 \times 10^{-9}$ \\
			\cmidrule{1-6}
			(2000, 200, 10, 0) & $2.41 \times 10^{-1}$ & $1.55 \times 10^{-8}$ & $1.55 \times 10^{-8}$ & $2.02 \times 10^{-1}$ & $1.55 \times 10^{-8}$ \\
			(2000, 200, 20, 0) & $1.78 \times 10^{-1}$ & $1.14 \times 10^{-8}$ & $1.14 \times 10^{-8}$ & $1.38 \times 10^{-1}$ & $1.14 \times 10^{-8}$ \\
			(2000, 200, 40, 0) & $1.33 \times 10^{-1}$ & $8.40 \times 10^{-9}$ & $8.40 \times 10^{-9}$ & $9.04 \times 10^{-2}$ & $8.40 \times 10^{-9}$ \\
			(2000, 200, 80, 0) & $1.01 \times 10^{-1}$ & $6.54 \times 10^{-9}$ & $6.54 \times 10^{-9}$ & $5.68 \times 10^{-2}$ & $6.54 \times 10^{-9}$ \\
			(2000, 200, 90, 0) & $9.72 \times 10^{-2}$ & $6.42 \times 10^{-9}$ & $6.42 \times 10^{-9}$ & $5.27 \times 10^{-2}$ & $6.42 \times 10^{-9}$ \\
			(2000, 200, 99, 0) & $9.32 \times 10^{-2}$ & $6.68 \times 10^{-9}$ & $6.68 \times 10^{-9}$ & $4.92 \times 10^{-2}$ & $6.68 \times 10^{-9}$ \\
			\cmidrule{1-6}
			(2000, 400, 10, 0) & $1.81 \times 10^{-1}$ & $9.01 \times 10^{-9}$ & $9.01 \times 10^{-9}$ & $1.37 \times 10^{-1}$ & $9.01 \times 10^{-9}$ \\
			(2000, 400, 20, 0) & $1.33 \times 10^{-1}$ & $6.55 \times 10^{-9}$ & $6.55 \times 10^{-9}$ & $8.95 \times 10^{-2}$ & $6.55 \times 10^{-9}$ \\
			(2000, 400, 40, 0) & $1.00 \times 10^{-1}$ & $4.88 \times 10^{-9}$ & $4.88 \times 10^{-9}$ & $5.71 \times 10^{-2}$ & $4.88 \times 10^{-9}$ \\
			(2000, 400, 80, 0) & $7.75 \times 10^{-2}$ & $4.00 \times 10^{-9}$ & $4.00 \times 10^{-9}$ & $3.38 \times 10^{-2}$ & $4.00 \times 10^{-9}$ \\
			(2000, 400, 90, 0) & $7.43 \times 10^{-2}$ & $4.13 \times 10^{-9}$ & $4.13 \times 10^{-9}$ & $3.08 \times 10^{-2}$ & $4.13 \times 10^{-9}$ \\
			(2000, 400, 99, 0) & $7.21 \times 10^{-2}$ & $4.33 \times 10^{-9}$ & $4.33 \times 10^{-9}$ & $2.82 \times 10^{-2}$ & $4.33 \times 10^{-9}$ \\
			\cmidrule{1-6}
			(2000, 800, 10, 0) & $1.34 \times 10^{-1}$ & $4.24 \times 10^{-9}$ & $4.24 \times 10^{-9}$ & $8.99 \times 10^{-2}$ & $4.24 \times 10^{-9}$ \\
			(2000, 800, 20, 0) & $1.00 \times 10^{-1}$ & $3.17 \times 10^{-9}$ & $3.17 \times 10^{-9}$ & $5.71 \times 10^{-2}$ & $3.17 \times 10^{-9}$ \\
			(2000, 800, 40, 0) & $7.80 \times 10^{-2}$ & $2.41 \times 10^{-9}$ & $2.41 \times 10^{-9}$ & $3.35 \times 10^{-2}$ & $2.41 \times 10^{-9}$ \\
			(2000, 800, 80, 0) & $6.18 \times 10^{-2}$ & $2.12 \times 10^{-9}$ & $2.12 \times 10^{-9}$ & $1.74 \times 10^{-2}$ & $2.12 \times 10^{-9}$ \\
			(2000, 800, 90, 0) & $5.97 \times 10^{-2}$ & $2.28 \times 10^{-9}$ & $2.28 \times 10^{-9}$ & $1.54 \times 10^{-2}$ & $2.28 \times 10^{-9}$ \\
			(2000, 800, 99, 0) & $5.77 \times 10^{-2}$ & $2.57 \times 10^{-9}$ & $2.57 \times 10^{-9}$ & $1.35 \times 10^{-2}$ & $2.57 \times 10^{-9}$ \\
			\cmidrule{1-6}
			(2000, 2000, 10, 0) & $9.24 \times 10^{-2}$ & $8.37 \times 10^{-12}$ & $8.37 \times 10^{-12}$ & $4.94 \times 10^{-2}$ & $9.11 \times 10^{-12}$ \\
			(2000, 2000, 20, 0) & $7.20 \times 10^{-2}$ & $3.46 \times 10^{-12}$ & $3.46 \times 10^{-12}$ & $2.79 \times 10^{-2}$ & $3.03 \times 10^{-12}$ \\
			(2000, 2000, 40, 0) & $5.75 \times 10^{-2}$ & $1.68 \times 10^{-12}$ & $1.68 \times 10^{-12}$ & $1.32 \times 10^{-2}$ & $1.77 \times 10^{-12}$ \\
			(2000, 2000, 80, 0) & $4.75 \times 10^{-2}$ & $2.13 \times 10^{-12}$ & $2.13 \times 10^{-12}$ & $2.71 \times 10^{-3}$ & $2.14 \times 10^{-12}$ \\
			(2000, 2000, 90, 0) & $4.61 \times 10^{-2}$ & $7.15 \times 10^{-12}$ & $7.15 \times 10^{-12}$ & $1.21 \times 10^{-3}$ & $7.03 \times 10^{-12}$ \\
			(2000, 2000, 99, 0) & $4.48 \times 10^{-2}$ & $9.47 \times 10^{-12}$ & $9.47 \times 10^{-12}$ & $1.19 \times 10^{-4}$ & $7.90 \times 10^{-12}$ \\
			\bottomrule
		\end{tabular}
	\end{table}

	\begin{table}[htbp]
		\centering
		\footnotesize 
		\setlength{\tabcolsep}{3pt} 
		\renewcommand{\arraystretch}{1.2} 
		\caption{Singular values of $A(col\_per, t)$ under different precisions for $t=1$.}
		\label{tb:ex1-2}
		\begin{tabular}{c c c c c c}
			\toprule
			\makecell{($m$, $n$, $col\_per$, $t$)} & 
			\makecell{Largest singular \\ value of \\ ${\rm double}(A)$} & 
			\makecell{Smallest positive \\ singular value of \\ ${\rm double}(A)$} & 
			\makecell{Smallest positive \\ singular value of \\ ${\rm double}({\rm single}(A))$} & 
			\makecell{Smallest positive \\ singular value of \\ ${\rm double}({\rm half}(A))$} & 
			\makecell{Smallest positive \\ singular value of \\ ${\rm double}({\rm bfloat16}(A))$} \\ 
			\midrule
			(2000, 100, 10, 1) & $3.36 \times 10^{-1}$ & $2.93 \times 10^{-8}$ & $2.93 \times 10^{-8}$ & $1.98 \times 10^{-4}$ & $2.93 \times 10^{-8}$ \\
			(2000, 100, 20, 1) & $2.46 \times 10^{-1}$ & $2.18 \times 10^{-8}$ & $2.18 \times 10^{-8}$ & $4.24 \times 10^{-5}$ & $2.18 \times 10^{-8}$ \\
			(2000, 100, 40, 1) & $1.80 \times 10^{-1}$ & $1.53 \times 10^{-8}$ & $1.53 \times 10^{-8}$ & $1.65 \times 10^{-6}$ & $1.53 \times 10^{-8}$ \\
			(2000, 100, 80, 1) & $1.33 \times 10^{-1}$ & $5.92 \times 10^{-5}$ & $5.92 \times 10^{-5}$ & $5.91 \times 10^{-5}$ & $5.92 \times 10^{-5}$ \\
			(2000, 100, 90, 1) & $1.27 \times 10^{-1}$ & $3.71 \times 10^{-6}$ & $3.71 \times 10^{-6}$ & $3.73 \times 10^{-6}$ & $3.70 \times 10^{-6}$ \\
			(2000, 100, 99, 1) & $1.22 \times 10^{-1}$ & $3.71 \times 10^{-3}$ & $3.71 \times 10^{-3}$ & $3.71 \times 10^{-3}$ & $3.72 \times 10^{-3}$ \\
			\cmidrule{1-6}
			(2000, 200, 10, 1) & $2.41 \times 10^{-1}$ & $2.02 \times 10^{-8}$ & $2.02 \times 10^{-8}$ & $4.27 \times 10^{-5}$ & $2.02 \times 10^{-8}$ \\
			(2000, 200, 20, 1) & $1.78 \times 10^{-1}$ & $1.45 \times 10^{-8}$ & $1.45 \times 10^{-8}$ & $8.75 \times 10^{-6}$ & $1.45 \times 10^{-8}$ \\
			(2000, 200, 40, 1) & $1.33 \times 10^{-1}$ & $1.03 \times 10^{-8}$ & $1.03 \times 10^{-8}$ & $2.37 \times 10^{-5}$ & $1.03 \times 10^{-8}$ \\
			(2000, 200, 80, 1) & $1.01 \times 10^{-1}$ & $7.39 \times 10^{-9}$ & $7.39 \times 10^{-9}$ & $2.60 \times 10^{-5}$ & $7.39 \times 10^{-9}$ \\
			(2000, 200, 90, 1) & $9.72 \times 10^{-2}$ & $7.10 \times 10^{-9}$ & $7.10 \times 10^{-9}$ & $2.07 \times 10^{-5}$ & $7.10 \times 10^{-9}$ \\
			(2000, 200, 99, 1) & $9.32 \times 10^{-2}$ & $7.29 \times 10^{-4}$ & $7.29 \times 10^{-4}$ & $7.29 \times 10^{-4}$ & $7.29 \times 10^{-4}$ \\
			\cmidrule{1-6}
			(2000, 400, 10, 1) & $1.81 \times 10^{-1}$ & $1.25 \times 10^{-8}$ & $1.25 \times 10^{-8}$ & $8.13 \times 10^{-6}$ & $1.25 \times 10^{-8}$ \\
			(2000, 400, 20, 1) & $1.33 \times 10^{-1}$ & $8.62 \times 10^{-9}$ & $8.62 \times 10^{-9}$ & $2.37 \times 10^{-5}$ & $8.62 \times 10^{-9}$ \\
			(2000, 400, 40, 1) & $1.00 \times 10^{-1}$ & $6.58 \times 10^{-9}$ & $6.58 \times 10^{-9}$ & $3.92 \times 10^{-5}$ & $6.58 \times 10^{-9}$ \\
			(2000, 400, 80, 1) & $7.75 \times 10^{-2}$ & $4.87 \times 10^{-9}$ & $4.87 \times 10^{-9}$ & $1.01 \times 10^{-5}$ & $4.87 \times 10^{-9}$ \\
			(2000, 400, 90, 1) & $7.43 \times 10^{-2}$ & $9.49 \times 10^{-5}$ & $9.49 \times 10^{-5}$ & $9.49 \times 10^{-5}$ & $9.49 \times 10^{-5}$ \\
			(2000, 400, 99, 1) & $7.21 \times 10^{-2}$ & $8.92 \times 10^{-5}$ & $8.92 \times 10^{-5}$ & $8.92 \times 10^{-5}$ & $8.91 \times 10^{-5}$ \\
			\cmidrule{1-6}
			(2000, 800, 10, 1) & $1.33 \times 10^{-1}$ & $6.35 \times 10^{-9}$ & $6.35 \times 10^{-9}$ & $2.36 \times 10^{-6}$ & $6.35 \times 10^{-9}$ \\
			(2000, 800, 20, 1) & $1.00 \times 10^{-1}$ & $4.70 \times 10^{-9}$ & $4.70 \times 10^{-9}$ & $1.39 \times 10^{-5}$ & $4.70 \times 10^{-9}$ \\
			(2000, 800, 40, 1) & $7.80 \times 10^{-2}$ & $3.50 \times 10^{-9}$ & $3.50 \times 10^{-9}$ & $1.06 \times 10^{-6}$ & $3.50 \times 10^{-9}$ \\
			(2000, 800, 80, 1) & $6.18 \times 10^{-2}$ & $2.95 \times 10^{-9}$ & $2.95 \times 10^{-9}$ & $3.29 \times 10^{-6}$ & $2.95 \times 10^{-9}$ \\
			(2000, 800, 90, 1) & $5.97 \times 10^{-2}$ & $2.93 \times 10^{-9}$ & $2.93 \times 10^{-9}$ & $1.89 \times 10^{-6}$ & $2.93 \times 10^{-9}$ \\
			(2000, 800, 99, 1) & $5.77 \times 10^{-2}$ & $7.70 \times 10^{-5}$ & $7.70 \times 10^{-5}$ & $7.70 \times 10^{-5}$ & $7.69 \times 10^{-5}$ \\
			\cmidrule{1-6}
			(2000, 2000, 10, 1) & $9.22 \times 10^{-2}$ & $6.07 \times 10^{-12}$ & $6.07 \times 10^{-12}$ & $4.58 \times 10^{-7}$ & $5.84 \times 10^{-12}$ \\
			(2000, 2000, 20, 1) & $7.20 \times 10^{-2}$ & $6.09 \times 10^{-12}$ & $6.09 \times 10^{-12}$ & $3.69 \times 10^{-6}$ & $5.89 \times 10^{-12}$ \\
			(2000, 2000, 40, 1) & $5.75 \times 10^{-2}$ & $1.56 \times 10^{-12}$ & $1.56 \times 10^{-12}$ & $9.42 \times 10^{-8}$ & $1.44 \times 10^{-12}$ \\
			(2000, 2000, 80, 1) & $4.75 \times 10^{-2}$ & $1.62 \times 10^{-11}$ & $1.62 \times 10^{-11}$ & $4.77 \times 10^{-7}$ & $1.70 \times 10^{-11}$ \\
			(2000, 2000, 90, 1) & $4.61 \times 10^{-2}$ & $1.66 \times 10^{-11}$ & $1.66 \times 10^{-11}$ & $2.54 \times 10^{-7}$ & $1.61 \times 10^{-11}$ \\
			(2000, 2000, 99, 1) & $4.48 \times 10^{-2}$ & $1.25 \times 10^{-6}$ & $1.25 \times 10^{-6}$ & $1.24 \times 10^{-6}$ & $1.21 \times 10^{-6}$ \\
			\bottomrule
		\end{tabular}
	\end{table}

	\begin{table}[htbp]
		\centering
		\footnotesize 
		\setlength{\tabcolsep}{3pt} 
		\renewcommand{\arraystretch}{1.2} 
		\caption{Singular values of $A(col\_per, t)$ under different precisions for $t=2$.}
		\label{tb:ex1-3}
		\begin{tabular}{c c c c c c}
			\toprule
			\makecell{($m$, $n$, $col\_per$, $t$)} & 
			\makecell{Largest singular \\ value of \\ ${\rm double}(A)$} & 
			\makecell{Smallest positive \\ singular value of \\ ${\rm double}(A)$} & 
			\makecell{Smallest positive \\ singular value of \\ ${\rm double}({\rm single}(A))$} & 
			\makecell{Smallest positive \\ singular value of \\ ${\rm double}({\rm half}(A))$} & 
			\makecell{Smallest positive \\ singular value of \\ ${\rm double}({\rm bfloat16}(A))$} \\ 
			\midrule
			(2000, 100, 10, 2) & $3.35 \times 10^{-1}$ & $1.99 \times 10^{-3}$ & $1.99 \times 10^{-3}$ & $1.99 \times 10^{-3}$ & $1.99 \times 10^{-3}$ \\
			(2000, 100, 20, 2) & $2.45 \times 10^{-1}$ & $1.87 \times 10^{-3}$ & $1.87 \times 10^{-3}$ & $1.87 \times 10^{-3}$ & $1.87 \times 10^{-3}$ \\
			(2000, 100, 40, 2) & $1.80 \times 10^{-1}$ & $6.19 \times 10^{-4}$ & $6.19 \times 10^{-4}$ & $6.19 \times 10^{-4}$ & $6.20 \times 10^{-4}$ \\
			(2000, 100, 80, 2) & $1.33 \times 10^{-1}$ & $1.61 \times 10^{-3}$ & $1.61 \times 10^{-3}$ & $1.61 \times 10^{-3}$ & $1.61 \times 10^{-3}$ \\
			(2000, 100, 90, 2) & $1.27 \times 10^{-1}$ & $1.94 \times 10^{-3}$ & $1.94 \times 10^{-3}$ & $1.94 \times 10^{-3}$ & $1.94 \times 10^{-3}$ \\
			(2000, 100, 99, 2) & $1.22 \times 10^{-1}$ & $4.89 \times 10^{-3}$ & $4.89 \times 10^{-3}$ & $4.89 \times 10^{-3}$ & $4.89 \times 10^{-3}$ \\
			\cmidrule{1-6}
			(2000, 200, 10, 2) & $2.40 \times 10^{-1}$ & $2.37 \times 10^{-4}$ & $2.37 \times 10^{-4}$ & $2.37 \times 10^{-4}$ & $2.37 \times 10^{-4}$ \\
			(2000, 200, 20, 2) & $1.78 \times 10^{-1}$ & $3.42 \times 10^{-4}$ & $3.42 \times 10^{-4}$ & $3.42 \times 10^{-4}$ & $3.42 \times 10^{-4}$ \\
			(2000, 200, 40, 2) & $1.33 \times 10^{-1}$ & $9.62 \times 10^{-5}$ & $9.62 \times 10^{-5}$ & $9.61 \times 10^{-5}$ & $9.62 \times 10^{-5}$ \\
			(2000, 200, 80, 2) & $1.01 \times 10^{-1}$ & $2.83 \times 10^{-4}$ & $2.83 \times 10^{-4}$ & $2.83 \times 10^{-4}$ & $2.82 \times 10^{-4}$ \\
			(2000, 200, 90, 2) & $9.72 \times 10^{-2}$ & $7.64 \times 10^{-4}$ & $7.64 \times 10^{-4}$ & $7.64 \times 10^{-4}$ & $7.63 \times 10^{-4}$ \\
			(2000, 200, 99, 2) & $9.32 \times 10^{-2}$ & $1.86 \times 10^{-3}$ & $1.86 \times 10^{-3}$ & $1.86 \times 10^{-3}$ & $1.86 \times 10^{-3}$ \\
			\cmidrule{1-6}
			(2000, 400, 10, 2) & $1.80 \times 10^{-1}$ & $1.07 \times 10^{-4}$ & $1.07 \times 10^{-4}$ & $1.07 \times 10^{-4}$ & $1.07 \times 10^{-4}$ \\
			(2000, 400, 20, 2) & $1.33 \times 10^{-1}$ & $1.64 \times 10^{-4}$ & $1.64 \times 10^{-4}$ & $1.64 \times 10^{-4}$ & $1.64 \times 10^{-4}$ \\
			(2000, 400, 40, 2) & $1.00 \times 10^{-1}$ & $6.38 \times 10^{-5}$ & $6.38 \times 10^{-5}$ & $6.38 \times 10^{-5}$ & $6.39 \times 10^{-5}$ \\
			(2000, 400, 80, 2) & $7.75 \times 10^{-2}$ & $1.87 \times 10^{-4}$ & $1.87 \times 10^{-4}$ & $1.87 \times 10^{-4}$ & $1.87 \times 10^{-4}$ \\
			(2000, 400, 90, 2) & $7.43 \times 10^{-2}$ & $1.59 \times 10^{-4}$ & $1.59 \times 10^{-4}$ & $1.59 \times 10^{-4}$ & $1.59 \times 10^{-4}$ \\
			(2000, 400, 99, 2) & $7.21 \times 10^{-2}$ & $4.26 \times 10^{-4}$ & $4.26 \times 10^{-4}$ & $4.26 \times 10^{-4}$ & $4.25 \times 10^{-4}$ \\
			\cmidrule{1-6}
			(2000, 800, 10, 2) & $1.33 \times 10^{-1}$ & $3.18 \times 10^{-5}$ & $3.18 \times 10^{-5}$ & $3.19 \times 10^{-5}$ & $3.18 \times 10^{-5}$ \\
			(2000, 800, 20, 2) & $1.00 \times 10^{-1}$ & $4.22 \times 10^{-5}$ & $4.22 \times 10^{-5}$ & $4.22 \times 10^{-5}$ & $4.22 \times 10^{-5}$ \\
			(2000, 800, 40, 2) & $7.79 \times 10^{-2}$ & $3.60 \times 10^{-5}$ & $3.60 \times 10^{-5}$ & $3.59 \times 10^{-5}$ & $3.59 \times 10^{-5}$ \\
			(2000, 800, 80, 2) & $6.18 \times 10^{-2}$ & $5.81 \times 10^{-5}$ & $5.81 \times 10^{-5}$ & $5.81 \times 10^{-5}$ & $5.82 \times 10^{-5}$ \\
			(2000, 800, 90, 2) & $5.97 \times 10^{-2}$ & $1.31 \times 10^{-4}$ & $1.31 \times 10^{-4}$ & $1.31 \times 10^{-4}$ & $1.31 \times 10^{-4}$ \\
			(2000, 800, 99, 2) & $5.77 \times 10^{-2}$ & $3.47 \times 10^{-4}$ & $3.47 \times 10^{-4}$ & $3.47 \times 10^{-4}$ & $3.47 \times 10^{-4}$ \\
			\cmidrule{1-6}
			(2000, 2000, 10, 2) & $9.20 \times 10^{-2}$ & $8.29 \times 10^{-12}$ & $8.29 \times 10^{-12}$ & $4.66 \times 10^{-9}$ & $8.37 \times 10^{-12}$ \\
			(2000, 2000, 20, 2) & $7.19 \times 10^{-2}$ & $9.54 \times 10^{-12}$ & $9.54 \times 10^{-12}$ & $4.04 \times 10^{-7}$ & $1.00 \times 10^{-11}$ \\
			(2000, 2000, 40, 2) & $5.75 \times 10^{-2}$ & $1.38 \times 10^{-11}$ & $1.38 \times 10^{-11}$ & $7.08 \times 10^{-10}$ & $1.47 \times 10^{-11}$ \\
			(2000, 2000, 80, 2) & $4.75 \times 10^{-2}$ & $2.15 \times 10^{-7}$ & $2.15 \times 10^{-7}$ & $2.15 \times 10^{-7}$ & $2.04 \times 10^{-7}$ \\
			(2000, 2000, 90, 2) & $4.61 \times 10^{-2}$ & $4.90 \times 10^{-7}$ & $4.90 \times 10^{-7}$ & $4.97 \times 10^{-7}$ & $5.13 \times 10^{-7}$ \\
			(2000, 2000, 99, 2) & $4.48 \times 10^{-2}$ & $1.37 \times 10^{-6}$ & $1.37 \times 10^{-6}$ & $1.37 \times 10^{-6}$ & $1.32 \times 10^{-6}$ \\
			\bottomrule
		\end{tabular}
	\end{table}

	\begin{table}[htbp]
		\centering
		\footnotesize 
		\setlength{\tabcolsep}{3pt} 
		\renewcommand{\arraystretch}{1.2}
		\caption{Singular values of $A(t, \Delta_t)$ under different precisions with $m=2000$ and $n=100$. }
		\label{tb:ex2}
		\begin{tabular}{cccccc}
			\toprule
			($t$, $\Delta_t$) & \makecell{Largest singular\\value of \\${\rm double}(A)$} & \makecell{Smallest positive\\singular value of\\${\rm double}(A)$} & \makecell{Smallest positive\\singular value of\\${\rm double}({\rm single}(A))$} & \makecell{Smallest positive\\singular value of\\${\rm double}({\rm half}(A))$} & \makecell{Smallest positive\\singular value of\\${\rm double}({\rm bfloat16}(A))$} \\
			\midrule
			$(1, 1)$ & $1.48 \times 10^{-1}$ & $2.67 \times 10^{-3}$ & $2.67 \times 10^{-3}$ & $2.67 \times 10^{-3}$ & $2.67 \times 10^{-3}$ \\
			$(1, 10^{-1})$ & $1.48 \times 10^{-1}$ & $2.67 \times 10^{-4}$ & $2.67 \times 10^{-4}$ & $2.67 \times 10^{-4}$ & $2.53 \times 10^{-4}$ \\
			$(1, 10^{-2})$ & $1.48 \times 10^{-1}$ & $2.67 \times 10^{-5}$ & $2.67 \times 10^{-5}$ & $2.50 \times 10^{-5}$ & $2.11 \times 10^{-5}$ \\
			\cmidrule{1-6}
			$(10^{-1}, 10^{-1})$ & $1.48 \times 10^{-1}$ & $2.67 \times 10^{-4}$ & $2.67 \times 10^{-4}$ & $2.67 \times 10^{-4}$ & $2.67 \times 10^{-4}$ \\
			$(10^{-1}, 10^{-2})$ & $1.48 \times 10^{-1}$ & $2.67 \times 10^{-5}$ & $2.67 \times 10^{-5}$ & $2.67 \times 10^{-5}$ & $2.63 \times 10^{-5}$ \\
			$(10^{-1}, 10^{-3})$ & $1.48 \times 10^{-1}$ & $2.67 \times 10^{-6}$ & $2.67 \times 10^{-6}$ & $2.63 \times 10^{-6}$ & $2.63 \times 10^{-6}$ \\
			\cmidrule{1-6}
			$(10^{-2}, 10^{-2})$ & $1.48 \times 10^{-1}$ & $2.67 \times 10^{-5}$ & $2.67 \times 10^{-5}$ & $2.67 \times 10^{-5}$ & $2.67 \times 10^{-5}$ \\
			$(10^{-2}, 10^{-3})$ & $1.48 \times 10^{-1}$ & $2.67 \times 10^{-6}$ & $2.67 \times 10^{-6}$ & $2.68 \times 10^{-6}$ & $2.80 \times 10^{-6}$ \\
			$(10^{-2}, 10^{-4})$ & $1.48 \times 10^{-1}$ & $2.67 \times 10^{-7}$ & $2.67 \times 10^{-7}$ & $2.88 \times 10^{-7}$ & $3.29 \times 10^{-7}$ \\
			\cmidrule{1-6}
			$(10^{-3}, 10^{-3})$ & $1.48 \times 10^{-1}$ & $2.67 \times 10^{-6}$ & $2.67 \times 10^{-6}$ & $2.68 \times 10^{-6}$ & $2.68 \times 10^{-6}$ \\
			$(10^{-3}, 10^{-4})$ & $1.48 \times 10^{-1}$ & $2.67 \times 10^{-7}$ & $2.67 \times 10^{-7}$ & $2.47 \times 10^{-7}$ & $2.68 \times 10^{-7}$ \\
			$(10^{-3}, 10^{-5})$ & $1.48 \times 10^{-1}$ & $2.67 \times 10^{-8}$ & $2.67 \times 10^{-8}$ & $4.12 \times 10^{-8}$ & $2.06 \times 10^{-8}$ \\
			\cmidrule{1-6}
			$(10^{-4}, 10^{-4})$ & $1.48 \times 10^{-1}$ & $2.67 \times 10^{-7}$ & $2.67 \times 10^{-7}$ & $2.88 \times 10^{-7}$ & $2.68 \times 10^{-7}$ \\
			$(10^{-4}, 10^{-5})$ & $1.48 \times 10^{-1}$ & $2.67 \times 10^{-8}$ & $2.67 \times 10^{-8}$ & $4.12 \times 10^{-8}$ & $2.70 \times 10^{-8}$ \\
			$(10^{-4}, 10^{-6})$ & $1.48 \times 10^{-1}$ & $2.67 \times 10^{-9}$ & $2.67 \times 10^{-9}$ & $4.12 \times 10^{-8}$ & $2.57 \times 10^{-9}$ \\
			\cmidrule{1-6}
			$(10^{-5}, 10^{-5})$ & $1.48 \times 10^{-1}$ & $2.67 \times 10^{-8}$ & $2.67 \times 10^{-8}$ & $7.77 \times 10^{-2}$ & $2.67 \times 10^{-8}$ \\
			$(10^{-5}, 10^{-6})$ & $1.48 \times 10^{-1}$ & $2.67 \times 10^{-9}$ & $2.67 \times 10^{-9}$ & $7.77 \times 10^{-2}$ & $2.73 \times 10^{-9}$ \\
			$(10^{-5}, 10^{-7})$ & $1.48 \times 10^{-1}$ & $2.67 \times 10^{-10}$ & $2.67 \times 10^{-10}$ & $7.77 \times 10^{-2}$ & $3.22 \times 10^{-10}$ \\
			\cmidrule{1-6}
			$(10^{-6}, 10^{-6})$ & $1.48 \times 10^{-1}$ & $2.67 \times 10^{-9}$ & $2.67 \times 10^{-9}$ & $7.77 \times 10^{-2}$ & $2.67 \times 10^{-9}$ \\
			$(10^{-6}, 10^{-7})$ & $1.48 \times 10^{-1}$ & $2.67 \times 10^{-10}$ & $2.67 \times 10^{-10}$ & $7.77 \times 10^{-2}$ & $2.61 \times 10^{-10}$ \\
			$(10^{-6}, 10^{-8})$ & $1.48 \times 10^{-1}$ & $2.67 \times 10^{-11}$ & $2.67 \times 10^{-11}$ & $7.77 \times 10^{-2}$ & $2.01 \times 10^{-11}$ \\
			\cmidrule{1-6}
			$(10^{-7}, 10^{-7})$ & $1.48 \times 10^{-1}$ & $2.67 \times 10^{-10}$ & $2.67 \times 10^{-10}$ & $7.77 \times 10^{-2}$ & $2.68 \times 10^{-10}$ \\
			$(10^{-7}, 10^{-8})$ & $1.48 \times 10^{-1}$ & $2.67 \times 10^{-11}$ & $2.67 \times 10^{-11}$ & $7.77 \times 10^{-2}$ & $2.64 \times 10^{-11}$ \\
			$(10^{-7}, 10^{-9})$ & $1.48 \times 10^{-1}$ & $2.67 \times 10^{-12}$ & $2.67 \times 10^{-12}$ & $7.77 \times 10^{-2}$ & $2.51 \times 10^{-12}$ \\
			\cmidrule{1-6}
			$(10^{-8}, 10^{-8})$ & $1.48 \times 10^{-1}$ & $2.67 \times 10^{-11}$ & $2.67 \times 10^{-11}$ & $7.77 \times 10^{-2}$ & $2.67 \times 10^{-11}$ \\
			$(10^{-8}, 10^{-9})$ & $1.48 \times 10^{-1}$ & $2.67 \times 10^{-12}$ & $2.67 \times 10^{-12}$ & $7.77 \times 10^{-2}$ & $2.67 \times 10^{-12}$ \\
			$(10^{-8}, 10^{-10})$ & $1.48 \times 10^{-1}$ & $7.77 \times 10^{-2}$ & $7.77 \times 10^{-2}$ & $7.77 \times 10^{-2}$ & $7.77 \times 10^{-2}$ \\
			\bottomrule
		\end{tabular}
	\end{table}
	
	\newpage

	\section{Some propositions}
	
	\begin{proposition}\label{pro:computedelk}
		For $M^k \in \R^{m\times n}$ with rank $r\leq m \leq n$, let $M^k = U\Sigma V^\top$ be its SVD decomposition, where $U \in \R^{m\times m}$, $\Sigma \in \R^{m\times n}$, and $V \in \R^{n\times n}$. $\Sigma$ is a diagonal matrix with diagonal entries $\Sigma_{ii} = \sigma_i$ for $1\leq i \leq r$, where $\{ \sigma_i  \}_{i=1}^r$ are the singular values with $\sigma_1\geq \cdots \geq \sigma_r>0$, and $\Sigma_{ii} = 0$ for $i>r$. Assume $\sigma_1 \geq a>0$. Then for any $0<c\leq 1$, the following properties hold. 
		
		(i) There exists an integer $\tilde r \leq r$ such that $\sigma_i \geq \frac{c\sigma_1}{r}$ for $i\leq \tilde r$ and $\sigma_i < \frac{c\sigma_1}{r}$ for the rest $i$. 
		
		(ii)The rescaled singular value $\frac{\sigma_i}{\sqrt{\sum_{i=1}^r \sigma_i^2} + \epsilon_{ns}}$ in (\ref{eq:sv-rmk}) is lower-bounded by $\frac{ca}{r(\sqrt{r}a + \epsilon_{ns})}$ for $i \leq \tilde r$. 
		
		(iii) For the above lower bound $l=\frac{ca}{r(\sqrt{r}a + \epsilon_{ns})}$ and upper bound $u=1$, suppose an iterative algorithm produces ${\hat D}^k = -U{\hat \Lambda}V^\top$ such that $|{\hat \Lambda}_{ii} - 1| \leq {\tilde \delta}_k$ for $i \leq {\tilde r}$, and $0\leq {\hat \Lambda}_{ii} \leq 1+ (c+{\tilde \delta}_k)/(1+c)$ for the rest $i$, with $\tilde \delta_k <1$. Then the conditions
		$$
		\|{\hat D}^k \| \leq 1+\delta_k,  \quad  \langle M^k, {\hat D}^k \rangle \leq -(1-\delta_k) \|M^k\|_\star
		$$ 
		are satisfied with $\delta_k = (c+{\tilde \delta}_k)/(1+c)$, where $\|\cdot\|$ denotes the spectral norm and $\|\cdot\|_\star$ the nuclear norm. 
	\end{proposition}
	
	\begin{proof}
		
		(i) The property holds since either $\tilde r = m$ such that $\sigma_i \geq \frac{c\sigma_1}{r}$ for all $i\leq m$, or there exists an integer $\tilde r <m$ such that $\sigma_{\tilde r} \geq \frac{c\sigma_1}{r}$ and $\sigma_{{\tilde r}+1} < \frac{c\sigma_1}{r}$. 
		
		(ii) Since $\sigma_i \geq \frac{c\sigma_1}{r}$ for $i\leq \tilde r$ from (i), we have 
		$$
		\frac{\sigma_i}{\sqrt{\sum_{i=1}^r \sigma_i^2} + \epsilon_{ns}} \geq \frac{\sigma_i}{\sqrt{r}\sigma_1 + \epsilon_{ns}} \geq \frac{c}{r} \cdot \frac{\sigma_1}{\sqrt{r}\sigma_1 + \epsilon_{ns}} \geq \frac{ca}{r(\sqrt{r}a + \epsilon_{ns})},  
		$$
		for $i\leq \tilde r$. 
		
		(iii) First, since $0 < 1- {\tilde \delta}_k \leq {\hat \Lambda}_{ii} \leq 1 + {\tilde \delta}_k < 1 + \delta_k$ for $i\leq \tilde r$, and  $0\leq {\hat \Lambda}_{ii} \leq 1 + \delta_k$ for the rest $i$, we know $\|{\hat D}^k\| = \|{\hat D}^k\|_2 = \max_i\{ {\hat \Lambda}_{ii} \} \leq 1+\delta_k$. 
		
		It remains to prove the inner-product bound. From (i), we have 
		$$
		\|M^k\|_\star = \sum_{i=1}^r \sigma_i = \sum_{i=1}^{\tilde r} \sigma_i + \sum_{\tilde r + 1}^r \sigma_i \leq \sum_{i=1}^{\tilde r} \sigma_i + c \sigma_1 \leq (1+c) \sum_{i=1}^{\tilde r} \sigma_i. 
		$$
		Then we arrive at 
		$$
		\langle M^k, {\hat D}^k \rangle = - \sum_{i=1}^r \sigma_i {\hat \Lambda}_{ii}  
		\leq - (1 - {\tilde \delta}_k) \sum_{i=1}^{\tilde r} \sigma_i 
		\leq - \frac{1 - {\tilde \delta}_k}{1+c} \|M^k\|_\star 
		= - (1-\delta_k) \|M^k\|_\star,
		$$
		where the last equality follows from $\delta_k = (c+{\tilde \delta}_k)/(1+c)$. This completes the proof. 
		
	\end{proof}

	\begin{proposition}\label{pro:ptv}
		Under the premise of Theorem 4.1 in \citep{amsel2025polar}, for $t= 1, \dots, T$, we have $0\leq p_t(x) \leq l_{t+1}$ for $x \in [0, l_t]$. 
	\end{proposition}
	
	\begin{proof}
		From Lemma A.1 and Theorem 4.1 in \citep{amsel2025polar}, we know $l_2 = p_1(l_1) = p_1(l)$, and $p_1(x)$ has at least $q$ distinct extremums on $(l, 1)$,  which implies that the equation $\frac{dp_1(x)}{dx} =0$ has at least $q$ distinct solutions on $(l, 1)$. Noticing that $\frac{dp_1(x)}{dx}$ is an even degree polynomial with degree at most $2q$, we know that $\frac{dp_1(x)}{dx} = 0$ has at most $q$ distinct solutions on $[0, \infty)$, and hence it has exactly $q$ distinct solutions on $[0, \infty)$ and the solutions lie in $(l, 1)$. Specifically, we know $\frac{dp_1(x)}{dx} = 0$ has no solution on $[0, l]$, and thus for $x \in [0, l]$, either $\frac{dp_1(x)}{dx}>0$ or $\frac{dp_1(x)}{dx} <0$. Then, from $p_1(0)=0$ and $p_1(l) = l_2\geq 0$, we have that $p_1(x)$ is increasing on $[0, l]$. Therefore $p_1(x) \in [0, l_2]$ on $[0, l]$. Repeating the above analysis $T$ times, we can get $p_t(x) \in [0, l_{t+1}]$ on $[0, l_t]$ for $t=1, ..., T$.

	\end{proof}

	\newpage
	
	\section{Proof of Theorem \ref{th:Inexact-gluon}}

	First, from the first condition in Assumption \ref{as:inexactLMO-layer} or Assumption \ref{as:inexactLMO-layer-2}, we can obtain an upper bound for $\|X_i^{k+1} - X_i^k\|_{(i)}$, that is 
	\begin{equation}\label{eq:diffXik-bd}
		\|X_i^{k+1} - X_i^k\|_{(i)} = t_i \eta \|{\hat D}_i^k\|_{(i)} \leq (1+\delta_{k, i}) t_i \eta. 
	\end{equation}
	
	Under Assumption \ref{as:L0L1smooth}, similar to the proof of Theorem 8 in \citep{riabinin2025gluon}, we can get  
	\begin{eqnarray}
		f(X^{k+1}) &\leq& f(X^k) + \sum_{i=1}^p \left[  \langle M_i^k, X_i^{k+1} - X_i^k \rangle + \|\nabla_i f(X^k) - M_i^k\|_{(i)\star} \|X_i^{k+1} - X_i^k\|_{(i)}   \right] \nonumber \\ 
		&& + \sum_{i=1}^p \frac{L_i^0 + L_i^1 \|\nabla_i f(X^k)\|_{(i)\star}}{2} \|X_i^{k+1} - X_i^k\|_{(i)}^2. \label{eq:fXk+1<}
	\end{eqnarray}
	
	Now we estimate $\sum_{i=1}^p \langle M_i^k, X_i^{k+1} - X_i^k \rangle$. For $i$ such that (\ref{eq:hatdk-obj-layer}) holds, i.e., $i \notin \I^k$, we have 
	\begin{eqnarray*}
		\langle M_i^k, X_i^{k+1} - X_i^k \rangle &=& t_i \eta \langle M_i^k, {\hat D}_i^k \rangle \\ 
		&\overset{(\ref{eq:hatdk-obj-layer})}{\leq}&  -(1-\delta_{k, i})t_i \eta \|M_i^k\|_{(i)\star} \\ 
		&\leq& -(1-\delta_{k, i})t_i \eta \|\nabla_i f(X^k)\|_{(i)\star}  + (1-\delta_{k, i})t_i \eta \|M_i^k - \nabla_i f(X^k)\|_{(i)\star}, 
	\end{eqnarray*}
	where the last inequality comes from the triangle inequality. 
	
	\noindent For $i \in \I^k$, under Assumption \ref{as:inexactLMO-layer}, we know $\|M_i^k\|_{(i)\star} \leq \epsilon_{1, i}$. Thus, 
	\begin{eqnarray*}
		&& \langle M_i^k, X_i^{k+1} - X_i^k \rangle \\ 
		&\leq& \|X_i^{k+1}- X_i^k\|_{(i)} \|M_i^k\|_{(i)\star} \\ 
		&\overset{(\ref{eq:diffXik-bd})}{\leq}& (1+\delta_{k, i}) t_i \eta \|M_i^k\|_{(i)\star}  \\ 
		&=&  -(1-\delta_{k, i})t_i \eta \|M_i^k\|_{(i)\star} + 2t_i \eta \|M_i^k\|_{(i)\star}   \\ 
		&\leq& -(1-\delta_{k, i})t_i \eta \|\nabla_i f(X^k)\|_{(i)\star}  + (1-\delta_{k, i})t_i \eta \|M_i^k - \nabla_i f(X^k)\|_{(i)\star} + 2t_i \eta \|M_i^k\|_{(i)\star} \\ 
		&\leq& -(1-\delta_{k, i})t_i \eta \|\nabla_i f(X^k)\|_{(i)\star}  + (1-\delta_{k, i})t_i \eta \|M_i^k - \nabla_i f(X^k)\|_{(i)\star} + 2t_i \eta \epsilon_{1, i}, 
	\end{eqnarray*}
	where the first inequality is by the Cauchy-Schwarz inequality, and the third inequality comes from the triangle inequality.  For $i \in \I^k$, under Assumption \ref{as:inexactLMO-layer-2}, we have  $\|M_i^k\|_{(i)\star} \leq \epsilon_{1, i}$ and $\langle M_i^k, {\hat D}_i^k \rangle \leq 0$. Hence, 
	\begin{eqnarray*}
		&& \langle M_i^k, X_i^{k+1} - X_i^k \rangle \\ 
		&\leq& 0 \\
		&=&  -(1-\delta_{k, i})t_i \eta \|M_i^k\|_{(i)\star} + (1-\delta_{k, i})t_i \eta \|M_i^k\|_{(i)\star}   \\ 
		&\leq& -(1-\delta_{k, i})t_i \eta \|\nabla_i f(X^k)\|_{(i)\star}  + (1-\delta_{k, i})t_i \eta \|M_i^k - \nabla_i f(X^k)\|_{(i)\star} + (1-\delta_{k, i})t_i \eta \|M_i^k\|_{(i)\star} \\ 
		&\leq& -(1-\delta_{k, i})t_i \eta \|\nabla_i f(X^k)\|_{(i)\star}  + (1-\delta_{k, i})t_i \eta \|M_i^k - \nabla_i f(X^k)\|_{(i)\star} + (1-\delta_{k, i})t_i \eta \epsilon_{1, i} \\ 
		&=& -(1-\delta_{k, i})t_i \eta \|\nabla_i f(X^k)\|_{(i)\star}  + (1-\delta_{k, i})t_i \eta \|M_i^k - \nabla_i f(X^k)\|_{(i)\star} + (1-\delta)t_i \eta \epsilon_{1, i} 
	\end{eqnarray*}
	where the second inequality comes from the triangle inequality.
	
	\noindent Combining the above two cases, we arrive at 
	\begin{eqnarray}
		\sum_{i=1}^p \langle M_i^k, X_i^{k+1} - X_i^k \rangle &\leq& \sum_{i=1}^p \left[   -(1-\delta_{k, i})t_i \eta \|\nabla_i f(X^k)\|_{(i)\star}  + (1-\delta_{k, i})t_i \eta \|M_i^k - \nabla_i f(X^k)\|_{(i)\star}   \right]  \nonumber \\ 
		&&  +  \sum_{i\in \I^k} \phi t_i \eta \epsilon_{1, i}. \label{eq:sumMik-Ik}
	\end{eqnarray}
	
	From (\ref{eq:diffXik-bd}), (\ref{eq:fXk+1<}), and (\ref{eq:sumMik-Ik}), we have 
	\begin{eqnarray*}
		f(X^{k+1}) &\leq& f(X^k) + \sum_{i=1}^p \left[  -(1-\delta_{k, i})t_i \eta \|\nabla_i f(X^k)\|_{(i)\star}  + 2 t_i \eta \|M_i^k - \nabla_i f(X^k)\|_{(i)\star}  \right] \\ 
		&&  + \sum_{i=1}^p \frac{L_i^0 + L_i^1 \|\nabla_i f(X^k)\|_{(i)\star}}{2} (1+\delta_{k, i})^2 t_i^2 \eta^2  +   \sum_{i\in \I^k} \phi t_i \eta\epsilon_{1, i}. 
	\end{eqnarray*}
	
	Taking expectations yields that 
	\begin{eqnarray*}
		\E[f(X^{k+1})] &\leq& \E[f(X^k)] + \sum_{i=1}^p \left[  -(1-\delta_{k, i})t_i \eta \E[\|\nabla_i f(X^k)\|_{(i)\star}]  + 2 t_i \eta \E[\|M_i^k - \nabla_i f(X^k)\|_{(i)\star}]  \right] \\ 
		&&  + \sum_{i=1}^p \frac{L_i^0 + L_i^1 \E[\|\nabla_i f(X^k)\|_{(i)\star}] }{2} (1+\delta_{k, i})^2 t_i^2 \eta^2  +   \sum_{i\in \I^k} \phi t_i \eta\epsilon_{1, i}. 
	\end{eqnarray*}
	
	Telescoping the above inequality gives 
	\begin{eqnarray}
		&&\sum_{i=1}^p \sum_{k=0}^{K-1} t_i \eta \E \left[  \| \nabla_i f(X^k)\|_{(i)\star} \right] (1-\delta_{k, i})   \nonumber  \\ 
		& \leq& \Delta^0 + \sum_{i=1}^p \left[  2 \sum_{k=0}^{K-1} t_i \eta \E \left[  \|M_i^k - \nabla_i f(X^k)\|_{(i)\star}  +  \sum_{k=0}^{K-1} \frac{(1+\delta_{k, i})^2 L_i^0}{2} t_i^2 \eta^2   \right] \right.  \nonumber \\ 
		&& \hskip 20mm \left.  + \sum_{k=0}^{K-1} \frac{(1+\delta_{k, i})^2L_i^1 t_i^2 \eta^2}{2} \E \left[  \| \nabla_i f(X^k)\|_{(i)\star}  \right] \right]  + \sum_{k=0}^{K-1} \sum_{i\in \I^k} \phi t_i \eta\epsilon_{1, i}, \label{eq:sumfgrad-InexactGluon}
	\end{eqnarray}
	where $\Delta^0 \eqdef f(X^0) - \inf_{X \in {\cal S}} f(X)$. 
	
	We introduce the following notation: $\mu_i^k \eqdef M_i^k - \nabla_i f(X^k)$, $\gamma_i^k \eqdef \nabla_i f_{\xi^k} (X^k) - \nabla_i f(X^k)$, $\alpha \eqdef 1- \beta$, and $S_i^k \eqdef \nabla_i f(X^{k-1}) - \nabla_i f(X^k)$. Then we have 
	\begin{eqnarray*}
		\mu_i^k &=& M_i^k - \nabla_i f(X^k) \\ 
		&=& \alpha \gamma_i^k + \beta S_i^k + \beta \mu_i^{k-1} \\ 
		&=& \beta^k \mu_i^0 + \sum_{\tau =1}^k \beta^{k-\tau} \alpha \gamma_i^\tau + \sum_{\tau=1}^k \beta^{k+1-\tau} S_i^\tau. 
	\end{eqnarray*}
	
	Hence we can obtain 
	\begin{eqnarray*}
		&& \E \left[  \| M_i^k - \nabla_i f(X^k)\|_{(i)\star}  \right] \\ 
		&=& \E \left[  \|\mu_i^k\|_{(i)\star}  \right] \\ 
		&\overset{(a)}{\leq}& \beta^k \E \left[  \|\mu_i^0\|_{(i)\star}  \right] + \E \left[  \| \sum_{\tau =1}^k \beta^{k-\tau} \alpha \gamma_i^\tau\|_{(i)\star} \right] + \sum_{\tau=1}^k \beta^{k+1-\tau} \E \left[  \|S_i^\tau \|_{(i)\star} \right] \\ 
		&\overset{(b)}{\leq}& \beta^k \rho \E \left[  \|\mu_i^0\|_{\rm F}  \right] + \rho \E \left[  \|\sum_{\tau=1}^k \beta^{k-\tau} \alpha \gamma_i^\tau \|_{\rm F} \right] + \sum_{\tau=1}^k \beta^{k+1-\tau} \left(  L_i^0 + L_i^1 \E \left[  \|\nabla_i f(X^\tau)\|_{(i)\star}  \right]  \right) t_i \eta (1+\delta_{\tau-1, i}) \\ 
		&\overset{(c)}{\leq}& \beta^k \rho \sqrt{\E \left[  \|\mu_i^0\|_{\rm F}^2  \right]} + \rho \sqrt{\E \left[  \sum_{\tau=1}^k\| \beta^{k-\tau} \alpha \gamma_i^\tau \|_{\rm F}^2 \right]} + \frac{(1+\delta)L_i^0t_i\eta}{\alpha} \\ 
		&& + (1+\delta)L_i^1 t_i \eta \sum_{\tau=1}^k \beta^{k+1-\tau} \E \left[  \|\nabla_i f(X^\tau)\|_{(i)\star}  \right] \\ 
		&\overset{(d)}{\leq}& \beta^k \rho \sigma + \alpha \rho \sigma \sqrt{\sum_{\tau=1}^k \beta^{2k-2\tau}} +  \frac{(1+\delta)L_i^0t_i\eta}{\alpha} + (1+\delta) L_i^1 t_i \eta \sum_{\tau=1}^k \beta^{k+1-\tau} \E \left[  \|\nabla_i f(X^\tau)\|_{(i)\star}  \right] \\ 
		&{\leq}& (1-\alpha)^k \rho\sigma + \sqrt{\alpha}\rho\sigma + \frac{(1+\delta)L_i^0t_i\eta}{\alpha} +  (1+\delta) L_i^1 t_i \eta \sum_{\tau=1}^k \beta^{k+1-\tau} \E \left[  \|\nabla_i f(X^\tau)\|_{(i)\star}  \right], 
	\end{eqnarray*}
	where (a) uses the triangle inequality, (b) uses Assumptions \ref{as:L0L1smooth}, \ref{as:rho}, and (\ref{eq:diffXik-bd}), (c) uses Jensen’s inequality, the fact that samples $\xi^k \sim {\cal D}$ are i.i.d, and the unbiasedness property in Assumption \ref{as:boundedvariance}, (d) uses Assumption \ref{as:boundedvariance}. 
	
	Combining the above inequality, (\ref{eq:sumfgrad-InexactGluon}), and the fact that $\delta_{k, i} \leq \delta$ gives that 
	\begin{eqnarray*}
		&& \sum_{i=1}^p \sum_{k=0}^{K-1} t_i \eta \E \left[ \| \nabla_i f(X^k)\|_{(i)\star}  \right] (1-\delta) \\ 
		&\leq& \Delta^0 + \sum_{i=1}^p   \left[   \sum_{k=0}^{K-1} 2(1-\alpha)^kt_i\eta\rho \sigma + \sum_{k=0}^{K-1} 2\sqrt{\alpha} t_i \eta \rho \sigma + \sum_{k=0}^{K-1}  \frac{2(1+\delta)L_i^0 t_i^2 \eta^2}{\alpha}   \right. \\ 
		&& \quad \left. + \sum_{k=0}^{K-1} 2(1+\delta)L_i^1 t_i^2 \eta^2 \sum_{\tau=1}^k \beta^{k+1-\tau} \E \left[  \|\nabla_i f(X^\tau) \|_{(i)\star}  \right] \right.  \\ 
		&& \quad  \left. + \sum_{k=0}^{K-1} \frac{(1+\delta)^2 L_i^0 t_i^2 \eta^2}{2} + \sum_{k=0}^{K-1}  \frac{(1+\delta)^2 L_i^1 t_i^2 \eta^2}{2} \E \left[  \|\nabla_i f(X^k)\|_{(i)\star}  \right] \right] + \sum_{k=0}^{K-1} \sum_{i\in \I^k} \phi t_i \eta\epsilon_{1, i}. 
	\end{eqnarray*}
	Noticed that 
	$$
	\sum_{k=0}^{K-1} \sum_{\tau=1}^k \beta^{k+1-\tau} \E \left[  \|\nabla_i f(X^\tau) \|_{(i)\star}  \right]  = \sum_{\tau=1}^{K-1} \sum_{k=\tau}^{K-1} \beta^{k+1-\tau}  \E \left[  \|\nabla_i f(X^\tau) \|_{(i)\star}  \right]  \leq \frac{1}{\alpha} \sum_{k=0}^{K-1}  \E \left[  \|\nabla_i f(X^\tau) \|_{(i)\star}  \right]. 
	$$
	Then we have 
	\begin{eqnarray*}
		&& \sum_{i=1}^p \sum_{k=0}^{K-1} t_i \eta \E \left[ \| \nabla_i f(X^k)\|_{(i)\star}  \right] (1-\delta) \\ 
		&\leq& \Delta^0 + \sum_{i=1}^p  \left[   \frac{2t_i \eta \rho \sigma}{\alpha} + 2K\sqrt{\alpha} t_i \eta \rho \sigma + \frac{2(1+\delta)KL_i^0 t_i^2 \eta^2}{\alpha}  + \frac{(1+\delta)^2 K L_i^0 t_i^2 \eta^2}{2}   \right. \\ 
		&& \quad \left.  +  \sum_{k=0}^{K-1} \left(  \frac{2(1+\delta)}{\alpha} + \frac{(1+\delta)^2}{2}  \right) L_i^1 t_i^2 \eta^2 \E \left[  \|\nabla_i f(X^k) \|_{(i)\star} \right]   \right]   + \sum_{k=0}^{K-1} \sum_{i\in \I^k} \phi t_i \eta\epsilon_{1, i}. 
	\end{eqnarray*}
	
	Now we consider two options: (1) $L_i^1 = 0$ for all $i \in \{  1, ..., p  \}$ and (2) $L_i^1 \neq 0$, for all $i \in \{  1, ..., p  \}$.  
	
	{\bf Case 1:} $L_i^1 = 0$ for all $i \in \{  1, ..., p  \}$. In this case, 
	\begin{eqnarray*}
		&& \min_{k=0, ..., K-1} \sum_{i=1}^p t_i \E \left[  \|\nabla_i f(X^k) \|_{(i)\star} \right] \\ 
		&\leq&  \frac{1}{K} \sum_{k=0}^{K-1} \sum_{i=1}^p t_i \E \left[  \|\nabla_i f(X^k) \|_{(i)\star} \right] \\ 
		&\leq&  \frac{1}{1-\delta} \left[ \frac{\Delta^0}{\eta K} + \frac{2\sum_{i=1}^p t_i \rho \sigma}{\alpha K} + 2\sqrt{\alpha} \sum_{i=1}^p t_i \rho \sigma + \frac{2(1+\delta)\sum_{i=1}^p L_i^0t_i^2 \eta}{\alpha} \right. \\ 
		&& \left. \hskip 14mm  + \frac{(1+\delta)^2\sum_{i=1}^p L_i^0 t_i^2 \eta}{2}  + \sum_{i=1}^p \phi t_i\epsilon_{1, i}  \right]. 
	\end{eqnarray*}
	
	{\bf Case 2:} $L_i^1 \neq 0$, for all $i \in \{  1, ..., p  \}$. First we let $\frac{\eta}{\alpha} \leq \min_{i}  \frac{1}{6(1+\delta)L_i^1 t_i}$. Then $\left(  \frac{2(1+\delta)}{\alpha} + \frac{(1+\delta)^2}{2}  \right) L_i^1 t_i \eta \leq \frac{1}{2}$ for all $i$, and 
	\begin{eqnarray*}
		&& \min_{k=0, ..., K-1} \sum_{i=1}^p t_i \E \left[  \|\nabla_i f(X^k) \|_{(i)\star} \right] \\ 
		&\leq&  \frac{1}{K} \sum_{k=0}^{K-1} \sum_{i=1}^p t_i \E \left[  \|\nabla_i f(X^k) \|_{(i)\star} \right] \\ 
		&\leq&  \frac{1}{1-\delta} \left[ \frac{2\Delta^0}{\eta K} + \frac{4\sum_{i=1}^p t_i \rho \sigma}{\alpha K} + 4\sqrt{\alpha} \sum_{i=1}^p t_i \rho \sigma + \frac{4(1+\delta)\sum_{i=1}^p L_i^0t_i^2 \eta}{\alpha} \right. \\ 
		&& \left. \hskip 14mm  + {(1+\delta)^2\sum_{i=1}^p L_i^0 t_i^2 \eta}  + \sum_{i=1}^p 2 \phi t_i\epsilon_{1, i}  \right]. 
	\end{eqnarray*}

	\newpage 
	
	\section{Proofs for inexact Gluon in the deterministic case}

	\subsection{Proof of Theorem \ref{th:gluon-deter}}

	First, from the first condition in Assumption \ref{as:inexactLMO-layer-deter} or Assumption \ref{as:inexactLMO-layer-2-deter}, we can obtain an upper bound for $\|X_i^{k+1} - X_i^k\|_{(i)}$, that is 
	\begin{equation}\label{eq:diffXik-bd-deter}
		\|X_i^{k+1} - X_i^k\|_{(i)} = t_i^k \|{\hat D}_i^k\|_{(i)} \leq (1+\delta_{k, i}) t_i^k. 
	\end{equation}
	
	Under Assumption \ref{as:L0L1smooth}, from the Lemma 1 in \citep{riabinin2025gluon}, we can obtain 
	\begin{equation}\label{eq:fXk+1<-deter}
		f(X^{k+1})  \leq  f(X^k)  +  \sum_{i=1}^p \left[  \langle \nabla_i f(X^k), X_i^{k+1} - X_i^k \rangle +  \frac{L_i^0 + L_i^1 \|\nabla_i f(X^k) \|_{(i)\star}}{2} \|X_i^{k+1} - X_i^k\|_{(i)}^2  \right]. 
	\end{equation}
	
	For $i$ such that (\ref{eq:hatdk-obj-layer-deter}) holds, i.e., $i \notin \I_D^k$, we have 
	\begin{eqnarray*}
		&& \langle \nabla_i f(X^k), X_i^{k+1} - X_i^k \rangle +  \frac{L_i^0 + L_i^1 \|\nabla_i f(X^k) \|_{(i)\star}}{2} \|X_i^{k+1} - X_i^k\|_{(i)}^2 \\ 
		&\overset{(\ref{eq:diffXik-bd-deter})}{\leq}& -(1-\delta_{k, i}) t_i^k \|\nabla_i f(X^k)\|_{(i)\star} + \frac{L_i^0 + L_i^1 \|\nabla_i f(X^k) \|_{(i)\star}}{2} (1+\delta_{k, i})^2 (t_i^k)^2 \\ 
		&=& - \frac{(1-\delta)^2\|\nabla_i f(X^k)\|_{(i)\star}^2}{2(1+\delta)^2 (L_i^0 + L_i^1 \|\nabla_i f(X^k)\|_{(i)\star}) }, 
	\end{eqnarray*}
	where the last equality comes from $t_i^k =  \frac{(1-\delta)\|\nabla_i f(X^k)\|_{(i)\star}}{(1+\delta)^2 (L_i^0 + L_i^1 \|\nabla_i f(X^k)\|_{(i)\star}) }$. 
	\noindent For $i\in \I_D^k$, under Assumption \ref{as:inexactLMO-layer-deter}, we know $\|\nabla_i f(X^k)\|_{(i)\star} \leq \epsilon_{1, i}$. Then 
	\begin{eqnarray*}
		&& \langle \nabla_i f(X^k), X_i^{k+1} - X_i^k \rangle +  \frac{L_i^0 + L_i^1 \|\nabla_i f(X^k) \|_{(i)\star}}{2} \|X_i^{k+1} - X_i^k\|_{(i)}^2 \\ 
		&\leq& (1+\delta_{k, i}) t_i^k \|\nabla_i f(X^k)\|_{(i)\star} + \frac{L_i^0 + L_i^1 \|\nabla_i f(X^k) \|_{(i)\star}}{2} (1+\delta_{k, i})^2 (t_i^k)^2 \\ 
		&=& - \frac{(1-\delta)^2\|\nabla_i f(X^k)\|_{(i)\star}^2}{2(1+\delta)^2 (L_i^0 + L_i^1 \|\nabla_i f(X^k)\|_{(i)\star}) }  +  \frac{2(1-\delta)\|\nabla_i f(X^k)\|_{(i)\star}^2}{(1+\delta)^2 (L_i^0 + L_i^1 \|\nabla_i f(X^k)\|_{(i)\star}} \\ 
		&\leq& - \frac{(1-\delta)^2\|\nabla_i f(X^k)\|_{(i)\star}^2}{2(1+\delta)^2 (L_i^0 + L_i^1 \|\nabla_i f(X^k)\|_{(i)\star}) }  +  \frac{2(1-\delta)\epsilon_{1, i}^2}{(1+\delta)^2 (L_i^0 + L_i^1 \epsilon_{1, i})}, 
	\end{eqnarray*}
	where the first inequality uses the Cauchy-Schwarz inequality and (\ref{eq:diffXik-bd-deter}), and the first equality comes from the definition of $t_i^k$. For $i\in \I_D^k$, under Assumption \ref{as:inexactLMO-layer-2-deter}, we know $\|\nabla_i f(X^k)\|_{(i)\star} \leq \epsilon_{1, i}$ and $\langle \nabla_i f(X^k), {\hat D}_i^k \rangle \leq 0$. Then 
	\begin{eqnarray*}
		&& \langle \nabla_i f(X^k), X_i^{k+1} - X_i^k \rangle +  \frac{L_i^0 + L_i^1 \|\nabla_i f(X^k) \|_{(i)\star}}{2} \|X_i^{k+1} - X_i^k\|_{(i)}^2 \\ 
		&\overset{(\ref{eq:diffXik-bd-deter})}{\leq}& \frac{L_i^0 + L_i^1 \|\nabla_i f(X^k) \|_{(i)\star}}{2} (1+\delta_{k, i})^2 (t_i^k)^2 \\ 
		&=& - \frac{(1-\delta)^2\|\nabla_i f(X^k)\|_{(i)\star}^2}{2(1+\delta)^2 (L_i^0 + L_i^1 \|\nabla_i f(X^k)\|_{(i)\star}) }  +  \frac{(1-\delta)^2\|\nabla_i f(X^k)\|_{(i)\star}^2}{(1+\delta)^2 (L_i^0 + L_i^1 \|\nabla_i f(X^k)\|_{(i)\star}} \\ 
		&\leq& - \frac{(1-\delta)^2\|\nabla_i f(X^k)\|_{(i)\star}^2}{2(1+\delta)^2 (L_i^0 + L_i^1 \|\nabla_i f(X^k)\|_{(i)\star}) }  +  \frac{(1-\delta)^2\epsilon_{1, i}^2}{(1+\delta)^2 (L_i^0 + L_i^1 \epsilon_{1, i})}, 
	\end{eqnarray*}
	where the first equality comes from the definition of $t_i^k$. 
	
	Combining the above two cases, we arrive at 
	\begin{eqnarray*}
		&&\sum_{i=1}^p \left[ \langle \nabla_i f(X^k), X_i^{k+1} - X_i^k \rangle +  \frac{L_i^0 + L_i^1 \|\nabla_i f(X^k) \|_{(i)\star}}{2} \|X_i^{k+1} - X_i^k\|_{(i)}^2 \right] \\ 
		&\leq& - \sum_{i=1}^p \frac{(1-\delta)^2\|\nabla_i f(X^k)\|_{(i)\star}^2}{2(1+\delta)^2 (L_i^0 + L_i^1 \|\nabla_i f(X^k)\|_{(i)\star}) }  +  \sum_{i \in \I_D^k}\frac{\phi (1-\delta) \epsilon_{1, i}^2}{(1+\delta)^2 (L_i^0 + L_i^1 \epsilon_{1, i})}. 
	\end{eqnarray*}
	
	From the above inequality and (\ref{eq:fXk+1<-deter}), we can get 
	\begin{equation}\label{eq:fXk+1-Ik-deter}
		f(X^{k+1}) \leq  f(X^k) - \sum_{i=1}^p \frac{(1-\delta)^2\|\nabla_i f(X^k)\|_{(i)\star}^2}{2(1+\delta)^2 (L_i^0 + L_i^1 \|\nabla_i f(X^k)\|_{(i)\star}) }  +  \sum_{i \in \I_D^k}\frac{\phi (1-\delta)\epsilon_{1, i}^2}{(1+\delta)^2 (L_i^0 + L_i^1 \epsilon_{1, i})}. 
	\end{equation}
	Summing the above terms from $k=0$ to $K-1$, we can get 
	\begin{eqnarray}
		&& \sum_{k=0}^{K-1} \sum_{i=1}^p \frac{\|\nabla_i f(X^k)\|_{(i)\star}^2}{2 (L_i^0 + L_i^1 \|\nabla_i f(X^k)\|_{(i)\star}) } \nonumber \\ 
		&\leq& \frac{(1+\delta)^2 (f(X^0) - f(X^K))}{(1-\delta)^2}  +  \sum_{k=0}^{K-1}  \sum_{i \in \I_D^k}\frac{\phi \epsilon_{1, i}^2}{(1-\delta) (L_i^0 + L_i^1 \epsilon_{1, i})} \nonumber \\ 
		&\leq&  \frac{(1+\delta)^2 \Delta^0}{(1-\delta)^2} + K \sum_{i=1}^p \frac{\phi \epsilon_{1, i}^2}{(1-\delta) (L_i^0 + L_i^1 \epsilon_{1, i})}, \label{eq:sumnablaXk2-deter}
	\end{eqnarray}
	where $\Delta^0 \eqdef f(X^0) - \inf_{X \in {\cal S}} f(X)$. 
	
	Then same as the proof of Theorem 3 in \citep{riabinin2025gluon}, we estimate the lower bound on the left-hand side of (\ref{eq:sumnablaXk2-deter}) in two ways. \\
	
	\noindent 1. Denote $L_{\max}^1 \eqdef \max_{i=1, ..., p} L_i^1$. Then we have 
	$$
	\psi_1 \left(  \sum_{i=1}^p \|\nabla_i f(X^k) \|_{(i)\star}  \right) \leq \sum_{i=1}^p  \frac{\|\nabla_i f(X^k)\|_{(i)\star}^2}{2 (L_i^0 + L_{\max}^1 \|\nabla_i f(X^k)\|_{(i)\star}) }, 
	$$
	where $\psi_1(t) \eqdef \frac{t^2}{2(\sum_{i=1}^p L_i^0 + L_{\max}^1 t)}$. Combining (\ref{eq:sumnablaXk2-deter}) and using the fact that $\psi_1$ is increasing, we can get 
	\begin{eqnarray*}
		K\psi_1 \left(  \min_{k=0, ..., K-1} \sum_{i=1}^p \|\nabla_i f(X^k) \|_{(i)\star}   \right) &\leq& \sum_{k=0}^{K-1} \psi_1 \left(  \sum_{i=1}^p \|\nabla_i f(X^k) \|_{(i)\star}  \right) \\ 
		&\leq& \frac{(1+\delta)^2 \Delta^0}{(1-\delta)^2} + K \sum_{i=1}^p \frac{\phi \epsilon_{1, i}^2}{(1-\delta) (L_i^0 + L_i^1 \epsilon_{1, i})}, 
	\end{eqnarray*}
	which implies that 
	$$
	\min_{k=0, ..., K-1} \sum_{i=1}^p \|\nabla_i f(X^k) \|_{(i)\star}  \leq \psi_1^{-1} \left(   \frac{(1+\delta)^2 \Delta^0}{(1-\delta)^2K} +   \sum_{i=1}^p \frac{\phi \epsilon_{1, i}^2}{(1-\delta) (L_i^0 + L_i^1 \epsilon_{1, i})} \right), 
	$$
	where $\psi_1^{-1}$ is the inverse function of $\psi_1$, which exists since $\psi_1$ is an increasing function. Thus, to reach the precision $ \min_{k=0, ..., K-1} \sum_{i=1}^p \|\nabla_i f(X^k) \|_{(i)\star} \leq \epsilon$, it is sufficient to let 
	$$
	\frac{(1+\delta)^2 \Delta^0}{(1-\delta)^2K} \leq (1-{\tilde c})\psi_1(\epsilon), \quad \text{and} \quad \sum_{i=1}^p \frac{\phi \epsilon_{1, i}^2}{(1-\delta) (L_i^0 + L_i^1 \epsilon_{1, i})} \leq {\tilde c} \psi_1(\epsilon), 
	$$
	for any $0<\tilde c<1$. Hence, it is sufficient to choose 
	$$
	K = \left\lceil  \frac{2(1+\delta)^2 \sum_{i=1}^p L_i^0 \Delta^0}{(1-\delta)^2(1-\tilde c)\epsilon^2}  +  \frac{2(1+\delta)^2 L_{\max}^1 \Delta^0}{(1-\delta)^2 (1-\tilde c)\epsilon}  \right\rceil, 
	$$
	and $\epsilon_{1, i}$ such that $\sum_{i=1}^p \frac{\phi \epsilon_{1, i}^2}{(1-\delta) (L_i^0 + L_i^1 \epsilon_{1, i})} \leq \frac{{\tilde c}\epsilon^2}{2(\sum_{i=1}^p L_i^0 + L_{\max}^1 \epsilon)}$. \\ 
	
	\noindent 2. Alternatively, we have 
	$$
	\psi_2 \left(  \sum_{i=1}^p \frac{1}{L_i^1} \|\nabla_i f(X^k) \|_{(i)\star}  \right) \leq \sum_{i=1}^p  \frac{\|\nabla_i f(X^k)\|_{(i)\star}^2}{2 (L_i^0 + L_i^1 \|\nabla_i f(X^k)\|_{(i)\star}) }, 
	$$
	where $\psi_2(t) \eqdef \frac{t^2}{2\left(  \sum_{i=1}^p \frac{L_i^0}{(L_i^1)^2 } + t \right)}$. Combining (\ref{eq:sumnablaXk2-deter}) and using the fact that $\psi_2$ is increasing, we obtain 
	\begin{eqnarray*}
		K\psi_2 \left(  \min_{k=0, ..., K-1} \sum_{i=1}^p \frac{1}{L_i^1}\|\nabla_i f(X^k) \|_{(i)\star}   \right) &\leq& \sum_{k=0}^{K-1} \psi_2 \left(  \sum_{i=1}^p \frac{1}{L_i^1}\|\nabla_i f(X^k) \|_{(i)\star}  \right) \\ 
		&\leq& \frac{(1+\delta)^2 \Delta^0}{(1-\delta)^2} + K \sum_{i=1}^p \frac{\phi \epsilon_{1, i}^2}{(1-\delta) (L_i^0 + L_i^1 \epsilon_{1, i})}, 
	\end{eqnarray*}
	which indicates 
	$$
	\min_{k=0, ..., K-1} \sum_{i=1}^p \frac{1}{L_i^1}\|\nabla_i f(X^k) \|_{(i)\star}  \leq \psi_2^{-1} \left(   \frac{(1+\delta)^2 \Delta^0}{(1-\delta)^2K} +   \sum_{i=1}^p \frac{\phi \epsilon_{1, i}^2}{(1-\delta) (L_i^0 + L_i^1 \epsilon_{1, i})} \right), 
	$$
	where $\psi_2^{-1}$ is the inverse function of $\psi_2$. Therefore, to reach the precision 
	$$
	\min_{k=0, ..., K-1} \sum_{i=1}^p \frac{1/L_i^1}{\frac{1}{p} \sum_{i=1}^p1/L_i^1}\|\nabla_i f(X^k) \|_{(i)\star} \leq \epsilon, 
	$$
	it is sufficient to choose 
	$$
	K = \left\lceil   \frac{2(1+\delta)^2\Delta^0 (\sum_{i=1}^p \frac{L_i^0}{(L_i^1)^2})}{(1-\delta)^2(1-\tilde c)\epsilon^2 \left(  \frac{1}{p} \sum_{j=1}^p \frac{1}{L_j^1} \right)^2}   +  \frac{2(1+\delta)^2\Delta^0}{(1-\delta)^2(1-\tilde c) \epsilon \left(  \frac{1}{p} \sum_{j=1}^p \frac{1}{L_j^1} \right)} \right\rceil, 
	$$
	and $\epsilon_{1, i}$ such that $$\sum_{i=1}^p \frac{\phi \epsilon_{1, i}^2}{(1-\delta) (L_i^0 + L_i^1 \epsilon_{1, i})} \leq \frac{{\tilde c} \left(\frac{\epsilon}{p} \sum_{j=1}^p \frac{1}{L_j^1}\right)^2}{2\left(  \sum_{i=1}^p \frac{L_i^0}{(L_i^1)^2 } + \frac{\epsilon}{p} \sum_{j=1}^p \frac{1}{L_j^1} \right)}.$$

	\newpage
	
	\subsection{Proof of Theorem \ref{th:gluon-deter-pl}}

	1. For the $p=1$ case, when $\epsilon_{1, i} = \epsilon_{1,1} = \sqrt{2\mu \epsilon}$, if there exists some ${\tilde k} \in \{0, ..., K-1\}$ such that $\| \nabla f(X^{\tilde k})\|_{\star} \leq \epsilon_{1, 1}$, then $f(X^{\tilde k}) - f^\star \leq \epsilon$ from Assumption \ref{as:PL}, and the result follows. Otherwise, for all $k\in \{0, ..., K-1\}$, we have $\| \nabla f(X^k)\|_{\star} > \sqrt{2\mu \epsilon}$. Then under Assumption \ref{as:inexactLMO-layer-deter} or Assumption \ref{as:inexactLMO-layer-2-deter}, the inequality (\ref{eq:hatdk-obj-layer-deter}) holds and $\I_D^k = \emptyset$ for all $k$. Hence, from (\ref{eq:fXk+1-Ik-deter}), we obtain the descent inequality
	$$
	f(X^{k+1}) \leq  f(X^k) - \frac{(1-\delta)^2\|\nabla f(X^k)\|_{\star}^2}{2(1+\delta)^2 (L_1^0 + L_1^1 \|\nabla f(X^k)\|_{\star}) }. 
	$$
	Since the function $\frac{t^2}{a+bt}$ is increasing on $t\geq 0$ for any $a, b\geq 0$ with $a+b>0$, from Assumption \ref{as:PL}, we can get 
	\begin{eqnarray*}
		f(X^{k+1}) &\leq& f(X^k)  -  \frac{2\mu (1-\delta)^2(f(X^k) - f^\star) }{2(1+\delta)^2 (L_1^0 + \sqrt{2\mu (f(X^k) - f^\star)}L_i^1)} \\ 
		&\leq& f(X^k) - \frac{\mu (1-\delta)^2(f(X^k) - f^\star) }{(1+\delta)^2 (L_1^0 + \sqrt{2\mu \Delta^0}L_1^1)},
	\end{eqnarray*}
	where the last inequality comes from $f(X^k) - f^\star \leq f(X^0) - f^\star = \Delta^0$. Subtracting $f^\star$ from both sides of above inequality, we arrive at 
	$$
	f(X^{k+1}) - f^\star \leq \left(  1 -  \frac{\mu (1-\delta)^2 }{(1+\delta)^2 (L_1^0 + \sqrt{2\mu \Delta^0}L_1^1)}  \right) (f(X^k) - f^\star). 
	$$
	Therefore, $f(X^K) - f^\star \leq \epsilon$ for $K =  \lceil  \frac{(1+\delta)^2 (L_1^0 + \sqrt{2\mu \Delta^0}L_1^1)}{(1-\delta)^2 \mu} \ln \frac{\Delta^0}{\epsilon} \rceil$. \\

	\noindent 2. For any $k$, since $\sum_{i=1}^{p} \|\nabla_i f(X)\|_{(i)\star}^2 \geq 2\mu \left(f(X) - f^\star\right)$ from Assumption \ref{as:PL}, we know there exists $\{a_i\}$ such that $a_i\geq 0$, $\sum_{i=1}^p a_i =1$, and $\|\nabla_i f(X^k)\|_{(i)\star}^2 \geq 2a_i \mu \left(f(X^k) - f^\star\right)$. In fact, we can choose $a_i = \|\nabla_i f(X^k)\|_{(i)\star}^2/\sum_{j=1}^{p} \|\nabla_j f(X^k)\|_{(j)\star}^2$. Then, since the function $\frac{t^2}{a+bt}$ is increasing on $t\geq 0$ for any $a, b\geq 0$ with $a+b>0$, we can get 
	\begin{eqnarray}
		&&\sum_{i=1}^p \frac{(1-\delta)^2\|\nabla_i f(X^k)\|_{(i)\star}^2}{2(1+\delta)^2 (L_i^0 + L_i^1 \|\nabla_i f(X^k)\|_{(i)\star}) } \nonumber \\ 
		&\geq& \sum_{i=1}^p \frac{2(1-\delta)^2a_i \mu (f(X^k) - f^\star) }{2(1+\delta)^2 (L_i^0 + L_i^1\sqrt{2a_i \mu (f(X^k) - f^\star)}) } \nonumber \\
		&\geq&  \sum_{i=1}^p \frac{(1-\delta)^2a_i \mu (f(X^k) - f^\star) }{(1+\delta)^2 (L_i^0 + L_i^1\sqrt{2 \mu (f(X^k) - f^\star)}) }. \label{eq:sumLi0+Li1nabla-deter-pl}
	\end{eqnarray}
	
	Denote $\theta \eqdef \frac{(1-\delta)^2\mu}{(1+\delta)^2(L_{\max}^0 + L_{\max}^1 \sqrt{2\mu \max\{ \Delta^0, \epsilon\}})}$. Next, we will show $f(X^k) - f^\star \leq \max\{\Delta^0, \epsilon\}$ for all $k\in \{0, ..., K-1\}$ by induction. For $k=0$, it holds trivially. Assume it holds for $k$. Then from (\ref{eq:fXk+1-Ik-deter}) and the assumption that $\frac{1}{\theta} \sum_{i=1}^p \frac{\phi (1-\delta)\epsilon_{1, i}^2}{(1+\delta)^2 (L_i^0 + L_i^1 \epsilon_{1, i})} \leq {\tilde c}\epsilon$, we have

	\begin{eqnarray*}
		f(X^{k+1}) &\leq&  f(X^k) - \sum_{i=1}^p \frac{(1-\delta)^2\|\nabla_i f(X^k)\|_{(i)\star}^2}{2(1+\delta)^2 (L_i^0 + L_i^1 \|\nabla_i f(X^k)\|_{(i)\star}) }  +  \sum_{i \in \I_D^k} \frac{\phi (1-\delta)\epsilon_{1, i}^2}{(1+\delta)^2 (L_i^0 + L_i^1 \epsilon_{1, i})} \\ 
		&\overset{(\ref{eq:sumLi0+Li1nabla-deter-pl})}{\leq}& f(X^k) -  \sum_{i=1}^p \frac{(1-\delta)^2a_i \mu (f(X^k) - f^\star) }{(1+\delta)^2 (L_i^0 + L_i^1\sqrt{2 \mu (f(X^k) - f^\star)}) }  +  \sum_{i=1}^p \frac{\phi (1-\delta)\epsilon_{1, i}^2}{(1+\delta)^2 (L_i^0 + L_i^1 \epsilon_{1, i})} \\ 
		&\leq& f(X^k)  -  \sum_{i=1}^p \theta a_i (f(X^k) - f^\star)  +  \theta {\tilde c}\epsilon, 
	\end{eqnarray*}
	which implies that 
	$$
	f(X^{k+1}) - f^\star \leq (1-\theta) (f(X^k) - f^\star) +  \theta {\tilde c}\epsilon \leq \max\{  \Delta^0, \epsilon  \}. 
	$$
	Therefore, $f(X^k) - f^\star \leq \max\{ \Delta^0, \epsilon\}$ for all $k \in \{ 0, ..., K-1\}$. From the above analysis, we also have 
	\begin{eqnarray*}
		f(X^K) - f^\star &\leq& (1-\theta) (f(X^{K-1}) - f^\star) + \theta {\tilde c}\epsilon \\ 
		&\leq& (1-\theta)^K (f(X^0) - f^\star) + \sum_{k=0}^{K-1} (1-\theta)^k \theta {\tilde c}\epsilon\\ 
		&\leq&  (1-\theta)^K (f(X^0) - f^\star)  +  {\tilde c}\epsilon. 
	\end{eqnarray*}
	Thus, to reach the precision $f(X^K) - f^\star \leq \epsilon$, it is sufficient to choose 
	$$
	K = \left\lceil  \frac{(1+\delta)^2(L_{\max}^0 + L_{\max}^1 \sqrt{2\mu \max\{ \Delta^0, \epsilon\}})}{(1-\delta)^2\mu} \ln \frac{\Delta^0}{(1-\tilde c)\epsilon}   \right\rceil. 
	$$

	\newpage 
	
	\section{Proofs for inexact Gluon with weight decay for the star-convex case}

	\subsection{A technical lemma}

	\begin{lemma}\label{lm:Xkbound-sc}
		For the iterates in Algorithm \ref{alg:gluon-sc}, define 
		\begin{equation}\label{eq:X-sc}
			X = \beta X^* + (1-\beta) X^k. 
		\end{equation}
		Let Assumption \ref{as:inexactLMO-layer-deter} or Assumption \ref{as:inexactLMO-layer-2-deter} hold for Option 1, and Assumption \ref{as:inexactLMO-layer} or Assumption \ref{as:inexactLMO-layer-2} hold for Option 2, with $\delta_{k, i} = \delta_i<1$, and 
		\begin{equation}\label{eq:etabeta-sc}
			\beta \max_{i\in \{1, ..., p\}} \left\{  \frac{\|X_i^0\|_{(i)}}{(1+\delta_i) t_i},  \frac{\|X_i^*\|_{(i)}}{(1-s)(1-\delta_i)t_i}  \right\} \leq \eta, 
		\end{equation}
		where $s\in[0, 1)$ is a parameter. Define the constant $\phi_i=2$ if Assumption \ref{as:inexactLMO-layer} or  \ref{as:inexactLMO-layer-deter} holds, and $\phi_i = 1-\delta_i$ if Assumption \ref{as:inexactLMO-layer-2} or \ref{as:inexactLMO-layer-2-deter} holds for all $i$.
		Then the following inequalities hold for all $i$:
		\begin{eqnarray}
			&&\|X_i - (1-\beta) X_i^k\|_{(i)} \leq (1-\delta_i) t_i\eta,  \   \|X_i-X_i^k\|_{(i)} \leq 2 t_i\eta, \label{eq:Xkbound-sc-1} \\
			&&\|X_i-X_i^{k+1}\|_{(i)}\leq 2 t_i\eta,  \  \|X_i^{k+1} - X_i^k\|_{(i)} \leq 2(1+\delta_i) t_i\eta, \label{eq:Xkbound-sc-2} 
		\end{eqnarray} 
		and 
		\begin{equation}\label{eq:innerdiff-f-sc}
			\langle M^k, X^{k+1} - X \rangle  \leq  -\sum_{i=1}^p s t_i \eta (1-\delta_i) \|M_i^k\|_{(i)\star} +  \sum_{i=1}^p \phi_i t_i \eta \epsilon_{1, i}. 
		\end{equation}
	\end{lemma}
	
	\begin{proof}
		
		First, we show the following inequality by induction
		\begin{equation}\label{eq:betaXk-sc}
			\beta \|X_i^k\|_{(i)} \leq (1+\delta_i) t_i\eta. 
		\end{equation}
		It is trivial that $\beta \|X_i^0\|_{(i)} \leq (1+\delta_i) t_i\eta$ from (\ref{eq:etabeta-sc}). Assume $\beta \|X_i^k\|_{(i)} \leq (1+\delta_i) t_i\eta$. Then 
		\begin{eqnarray*}
			\beta \|X_i^{k+1}\|_{(i)} &\leq& \beta \|X_i^{k+1} - (1-\beta)X_i^k\|_{(i)} + (1-\beta) \beta \|X_i^k\|_{(i)}  \\ 
			&\leq& \beta \| t_i\eta {\hat D}_i^k\|_{(i)}   +  (1-\beta)(1+\delta_i) t_i \eta \\ 
			&\leq& (1+\delta_i)\beta t_i \eta  +  (1-\beta)(1+\delta_i) t_i \eta \\ 
			&=& (1+\delta_i)t_i \eta,
		\end{eqnarray*}
		where the first inequality uses the triangle inequality, the second inequality comes from the update rule of $X_i^{k+1}$ and the induction hypothesis, and the last inequality uses Assumption \ref{as:inexactLMO-layer-deter} or \ref{as:inexactLMO-layer-2-deter} for Option 1, and Assumption \ref{as:inexactLMO-layer} or \ref{as:inexactLMO-layer-2} for Option 2. 
		
		Hence, we can obtain 
		\begin{equation}\label{eq:X-1-betaXk-sc}
			\|X_i-(1-\beta) X_i^k\|_{(i)}  \overset{(\ref{eq:X-sc})}{=} \beta\|X^*\|_{(i)}  \overset{(\ref{eq:etabeta-sc})}{\leq} (1-s)(1-\delta_i) t_i \eta \leq (1-\delta_i) t_i \eta, 
		\end{equation}
		and 
		$$
		\|X_i - X_i^k\|_{(i)}  \overset{(\ref{eq:X-sc})}{=}  \beta \|X_i^* - X_i^k\|_{(i)}  \leq \beta \|X_i^*\|_{(i)}  + \beta \|X_i^k\|_{(i)}  \leq (1-\delta_i)t_i \eta + (1+\delta_i)t_i \eta = 2t_i \eta, 
		$$
		where the last inequality uses (\ref{eq:etabeta-sc}) and the previously obtained inequality (\ref{eq:betaXk-sc}). Then, 
		\begin{eqnarray*}
			\|X_i - X_i^{k+1} \|_{(i)}  &\leq& \|X_i-(1-\beta)X_i^k\|_{(i)}  + \|X_i^{k+1} - (1-\beta) X_i^k\|_{(i)}  \\ 
			&\overset{(\ref{eq:X-1-betaXk-sc})}{\leq}& (1-\delta_i) t_i \eta +  \|X_i^{k+1} - (1-\beta) X_i^k\|_{(i)}  \\ 
			&=&   (1-\delta_i) t_i \eta +  t_i \eta \|{\hat D}_i^k\|_{(i)}  \\ 
			&\leq&  (1-\delta_i)t_i \eta + (1+\delta_i) t_i \eta \\
			&=& 2t_i \eta, 
		\end{eqnarray*}
		and 
		\begin{eqnarray*}
			\|X_i^{k+1} - X_i^k\|_{(i)}  &=& \|X_i^{k+1} - (1-\beta)X_i^k - \beta X_i^k\|_{(i)}  \\
			&\leq&  \|X_i^{k+1} - (1-\beta)X_i^k\|_{(i)}  + \beta \|X_i^k\|_{(i)}  \\ 
			&\overset{(\ref{eq:betaXk-sc})}{\leq}& t_i \eta  \|{\hat D}_i^k\|_{(i)}  +  (1+\delta_i) t_i \eta \\ 
			&\leq& 2(1+\delta_i) t_i \eta, 
		\end{eqnarray*}
		where the last  inequality uses Assumption \ref{as:inexactLMO-layer-deter} or \ref{as:inexactLMO-layer-2-deter} for Option 1, and Assumption \ref{as:inexactLMO-layer} or \ref{as:inexactLMO-layer-2} for Option 2.

		Next we estimate $\langle M^k, X^{k+1} - X \rangle$.	For the $X$ defined in (\ref{eq:X-sc}), i.e., $X = \beta X^* + (1-\beta) X^k$, we have 
		\begin{eqnarray}
			&& \langle M^k, X^{k+1} - X \rangle  \nonumber  \\
			&=& \langle M^k, X^{k+1} - (1-\beta)X^k - \beta X^* \rangle   \nonumber  \\ 
			&=& \sum_{i=1}^p \langle M_i^k, t_i \eta {\hat D}_i^k \rangle  -  \sum_{i=1}^p \langle M_i^k, \beta X_i^* \rangle  \nonumber   \\ 
			&\leq& \sum_{i=1}^p t_i \eta \langle M_i^k,  {\hat D}_i^k \rangle  +  \sum_{i=1}^p \beta \|X_i^*\|_{(i)} \cdot \|M_i^k\|_{(i)\star}  \nonumber  \\ 
			&\leq&  \sum_{i=1}^p t_i \eta \langle M_i^k,  {\hat D}_i^k \rangle   + \sum_{i=1}^p  (1-s)(1-\delta_i) t_i \eta \|M_i^k\|_{(i)\star}, \label{eq:innerdiff-sc}
		\end{eqnarray}
		where the first inequality comes from the Cauchy-Schwarz inequality. 
		
		Next we estimate $ \sum_{i=1}^p t_i \eta \langle M_i^k,  {\hat D}_i^k \rangle$. For $i$ such that (\ref{eq:hatdk-obj-layer}) or  (\ref{eq:hatdk-obj-layer-deter}) holds, we have 
		$$
		t_i \eta \langle M_i^k,  {\hat D}_i^k \rangle \leq -t_i \eta (1-\delta_i) \|M_i^k\|_{(i)\star}. 
		$$
		For $i$ such that (\ref{eq:hatdk-obj-layer}) or  (\ref{eq:hatdk-obj-layer-deter}) does not hold, under Assumption \ref{as:inexactLMO-layer} or \ref{as:inexactLMO-layer-deter}, we know $\|M_i^k \|_{(i)\star} \leq \epsilon_{1, i}$. Hence, 
		\begin{eqnarray*}
			t_i \eta \langle M_i^k,  {\hat D}_i^k \rangle &\leq& t_i \eta \|{\hat D}_i^k \|_{(i)} \cdot \|M_i^k \|_{(i)\star}  \\ 
			&\leq& t_i \eta (1+\delta_i) \|M_i^k \|_{(i)\star}  \\ 
			&=&  -t_i \eta (1-\delta_i) \|M_i^k\|_{(i)\star} +  2t_i \eta  \|M_i^k\|_{(i)\star} \\
			&\leq& -t_i \eta (1-\delta_i) \|M_i^k \|_{(i)\star} +  2t_i \eta \epsilon_{1, i}, 
		\end{eqnarray*}
		where the first inequality comes from the Cauchy-Schwarz inequality, and the second inequality is from (i) in Assumption \ref{as:inexactLMO-layer-deter}. For $i$ such that (\ref{eq:hatdk-obj-layer}) or  (\ref{eq:hatdk-obj-layer-deter}) does not hold, under Assumption \ref{as:inexactLMO-layer-2} or \ref{as:inexactLMO-layer-2-deter}, we know $\|M_i^k \|_{(i)\star} \leq \epsilon_{1, i}$ and $\langle M_i^k, {\hat D}_i^k \rangle \leq 0$. Thus, 
		\begin{eqnarray*}
			t_i \eta \langle M_i^k,  {\hat D}_i^k \rangle &\leq& 0  \\ 
			&=&  -t_i \eta (1-\delta_i) \|M_i^k\|_{(i)\star} +  t_i \eta (1-\delta_i) \|M_i^k\|_{(i)\star} \\
			&\leq& -t_i \eta (1-\delta_i) \|M_i^k\|_{(i)\star} +  (1-\delta_i) t_i \eta \epsilon_{1, i}. 
		\end{eqnarray*}
		Combining the above two cases, we arrive at 
		$$
		\sum_{i=1}^p t_i \eta \langle M_i^k,  {\hat D}_i^k \rangle \leq  -\sum_{i=1}^p t_i \eta (1-\delta_i) \|M_i^k\|_{(i)\star} +  \sum_{i=1}^p \phi_i t_i \eta \epsilon_{1, i}. 
		$$
		From the above inequality and (\ref{eq:innerdiff-sc}), we can obtain (\ref{eq:innerdiff-f-sc}). 
		
	\end{proof}

	\newpage

	\subsection{Proof of Theorem \ref{th:gluon-inexact-sc-deter}}

	First, we can upper-bound $f(X^{k+1})$ as follows. 
	\begin{eqnarray*}
		f(X^{k+1}) &\overset{(a)}{\leq}& f(X^k) + \langle \nabla f(X^k), X^{k+1} - X^k \rangle + \sum_{i=1}^p \frac{L_i^0 + L_i^1 \|\nabla_i f(X^k)\|_{(i)\star}}{2} \|X_i^{k+1} - X_i^k\|_{(i)}^2 \\ 
		&\overset{(\ref{eq:innerdiff-f-sc})}{\leq}& f(X^k) + \langle \nabla f(X^k), X-X^k \rangle  -\sum_{i=1}^p s t_i \eta (1-\delta_i) \|\nabla_i f(X^k)\|_{(i)\star} +  \sum_{i=1}^p \phi_i t_i \eta \epsilon_{1, i} \\ 
		&& +  \sum_{i=1}^p \frac{L_i^0 + L_i^1 \|\nabla_i f(X^k)\|_{(i)\star}}{2} \|X_i^{k+1} - X_i^k\|_{(i)}^2 \\ 
		&\overset{(b)}{\leq}& f(X) +   \sum_{i=1}^p \frac{L_i^0 + L_i^1 \|\nabla_i f(X^k)\|_{(i)\star}}{2} \|X_i - X_i^k\|_{(i)}^2 -\sum_{i=1}^p s t_i \eta (1-\delta_i) \|\nabla_i f(X^k)\|_{(i)\star} \\ 
		&&  +  \sum_{i=1}^p \phi_i t_i \eta \epsilon_{1, i}  +   \sum_{i=1}^p \frac{L_i^0 + L_i^1 \|\nabla_i f(X^k)\|_{(i)\star}}{2} \|X_i^{k+1} - X_i^k\|_{(i)}^2 \\  
		&\overset{(c)}{\leq}& f(X) - \sum_{i=1}^p s t_i \eta (1-\delta_i) \|\nabla_i f(X^k)\|_{(i)\star} +  \sum_{i=1}^p \phi_i t_i \eta \epsilon_{1, i} \\ 
		&& +  \sum_{i=1}^p  2t_i^2 \eta^2 (1 + (1+\delta_i)^2)(L_i^0 + L_i^1 \|\nabla_i f(X^k)\|_{(i)\star})  \\  
		&\overset{(d)}{\leq}& \beta f(X^*)  + (1-\beta) f(X^k)  - \sum_{i=1}^p t_i \eta \left(  s(1-\delta_i) - 2t_i \eta (1 + (1+\delta_i)^2)L_i^1  \right) \|\nabla_i f(X^k)\|_{(i)\star} \\ 
		&&  + \sum_{i=1}^p  2t_i^2 \eta^2 (1 + (1+\delta_i)^2) L_i^0  +   \sum_{i=1}^p \phi_i t_i \eta \epsilon_{1, i} \\ 
		&\overset{(e)}{\leq}&  \beta f(X^*)  + (1-\beta) f(X^k)  + \sum_{i=1}^p  2t_i^2 \eta^2 (1 + (1+\delta_i)^2) L_i^0  +   \sum_{i=1}^p \phi_i t_i \eta \epsilon_{1, i}, 
	\end{eqnarray*}
	where (a) and (b) use Assumption \ref{as:L0L1smooth} and Lemma 1 in \citep{riabinin2025gluon}, (c) uses (\ref{eq:Xkbound-sc-1}) and (\ref{eq:Xkbound-sc-2}) in Lemma \ref{lm:Xkbound-sc}, (d) uses the star-convexity in Assumption \ref{as:starconvex}, and (e) uses the assumption that $ s(1-\delta_i) \geq 2t_i \eta (1 + (1+\delta_i)^2)L_i^1$ for all $i$. Subtracting $f(X^*)$ on the both sides of the above inequality, we arrive at 
	$$
	f(X^{k+1}) - f(X^*) \leq  (1-\beta) (f(X^k) - f(X^*)) + \sum_{i=1}^p  2t_i^2 \eta^2 (1 + (1+\delta_i)^2) L_i^0  +   \sum_{i=1}^p \phi_i t_i \eta \epsilon_{1, i}, 
	$$
	which implies that 
	$$
	f(X^K) - f(X^*) \leq (1-\beta)^K (f(X^0) - f(X^*)) + \frac{1}{\beta}  \sum_{i=1}^p  2 (1 + (1+\delta_i)^2) L_i^0 t_i^2 \eta^2 + \sum_{i=1}^p \frac{\phi_i t_i \eta \epsilon_{1, i}}{\beta}. 
	$$

	\newpage 
	
	\subsection{Proof of Corollary \ref{co:inexact-gluon-sc-deter}}

	For $\beta = \frac{\ln K}{K}$, from the fact that $(1-x)^{\frac{1}{x}} \leq e^{-1}$ for any $0<x<1$, we have 
	\begin{eqnarray*}
		(1-\beta)^K &=& \left(  1 - \frac{\ln K}{K}  \right)^{\frac{K}{\ln K} \cdot \ln K} \\ 
		&\leq& e^{-\ln K} \\ 
		&=& \frac{1}{K}.  
	\end{eqnarray*}
	
	For the two cases, the conditions that $ s(1-\delta_i) \geq 2 (1 + (1+\delta_i)^2)L_i^1 t_i \eta$ for all $i$ are both satisfied. Hence, from Theorem \ref{th:gluon-inexact-sc-deter}, we can get the results.

	\newpage 
	
	\subsection{Proof of Theorem \ref{th:gluon-inexact-sc}}

	First, for the estimation of $\E [ \|M_i^k - \nabla_i f(X^k)\|_{(i)\star} ]$, similar to the proof of Theorem \ref{th:Inexact-gluon}, we can obtain 
	\begin{eqnarray}
		\E [ \|M_i^k - \nabla_i f(X^k)\|_{(i)\star} ] &\leq&  (1-\alpha)^k\rho \sigma + \sqrt{\alpha} \rho \sigma + \frac{2(1+\delta_i)L_i^0 t_i \eta}{\alpha}  \nonumber  \\ 
		&& +  2(1+\delta_i)L_i^1 t_i \eta \sum_{\tau=1}^k (1-\alpha)^{k+1-\tau} \E [\|\nabla_i f(X^\tau)\|_{(i)\star}], \label{eq:Mik-nablaf-sc}
	\end{eqnarray}
	by using $\|X_i^{k+1} - X_i^k\|_{(i)} \leq 2(1+\delta_i) t_i \eta$ in Lemma \ref{lm:Xkbound-sc}. It should be noticed that $\alpha$ and $\beta$ are independent parameters in Algorithm \ref{alg:gluon-sc}. Then we upper-bound $f(X^{k+1})$ as follows. 
	
	\begin{eqnarray*}
		f(X^{k+1}) &\overset{(a)}{\leq}& f(X^k) + \langle \nabla f(X^k), X^{k+1} - X^k \rangle + \sum_{i=1}^p \frac{L_i^0 + L_i^1 \|\nabla_i f(X^k)\|_{(i)\star}}{2} \|X_i^{k+1} - X_i^k\|_{(i)}^2 \\ 
		&=&  f(X^k) +  \langle \nabla f(X^k) - M^k, X^{k+1} - X^k \rangle  + \langle M^k,  X^{k+1} - X^k \rangle \\ 
		&&   + \sum_{i=1}^p \frac{L_i^0 + L_i^1 \|\nabla_i f(X^k)\|_{(i)\star}}{2} \|X_i^{k+1} - X_i^k\|_{(i)}^2 \\ 
		&\overset{(\ref{eq:innerdiff-f-sc})}{\leq}&  f(X^k) +  \langle \nabla f(X^k) - M^k, X^{k+1} - X^k \rangle +  \langle M^k,  X - X^k \rangle   +  \sum_{i=1}^p \phi_i t_i \eta \epsilon_{1, i} \\ 
		&&  -\sum_{i=1}^p s t_i \eta (1-\delta_i) \|M_i^k\|_{(i)\star}    + \sum_{i=1}^p \frac{L_i^0 + L_i^1 \|\nabla_i f(X^k)\|_{(i)\star}}{2} \|X_i^{k+1} - X_i^k\|_{(i)}^2  \\ 
		&=&  f(X^k) + \langle \nabla f(X^k), X - X^k \rangle + \langle M^k - \nabla f(X^k), X - X^{k+1} \rangle +  \sum_{i=1}^p \phi_i t_i \eta \epsilon_{1, i} \\ 
		&& -\sum_{i=1}^p s t_i \eta (1-\delta_i) \|M_i^k\|_{(i)\star}    + \sum_{i=1}^p \frac{L_i^0 + L_i^1 \|\nabla_i f(X^k)\|_{(i)\star}}{2} \|X_i^{k+1} - X_i^k\|_{(i)}^2  \\ 
		&\overset{(b)}{\leq}& f(X) +  \sum_{i=1}^p \frac{L_i^0 + L_i^1 \|\nabla_i f(X^k)\|_{(i)\star}}{2} \|X_i - X_i^k\|_{(i)}^2  + \langle M^k - \nabla f(X^k), X - X^{k+1} \rangle  \\ 
		&&  + \sum_{i=1}^p \phi_i t_i \eta \epsilon_{1, i} -\sum_{i=1}^p s t_i \eta (1-\delta_i) \|M_i^k\|_{(i)\star}    + \sum_{i=1}^p \frac{L_i^0 + L_i^1 \|\nabla_i f(X^k)\|_{(i)\star}}{2} \|X_i^{k+1} - X_i^k\|_{(i)}^2  \\ 
		&\overset{(c)}{\leq}& f(X) + \sum_{i=1}^p 2t_i \eta \| M_i^k - \nabla_i f(X^k)\|_{(i)\star}  +  \sum_{i=1}^p \phi_i t_i \eta \epsilon_{1, i}   -\sum_{i=1}^p s t_i \eta (1-\delta_i) \|M_i^k\|_{(i)\star} \\ 
		&& +  \sum_{i=1}^p  2t_i^2 \eta^2 (1 + (1+\delta_i)^2)(L_i^0 + L_i^1 \|\nabla_i f(X^k)\|_{(i)\star}), 
	\end{eqnarray*}
	where (a) and (b) use Assumption \ref{as:L0L1smooth} and Lemma 1 in \citep{riabinin2025gluon}, (c) uses the Cauchy-Schwarz inequality, (\ref{eq:Xkbound-sc-1}) and (\ref{eq:Xkbound-sc-2}) in Lemma \ref{lm:Xkbound-sc}. Then from the triangle inequality $\|M_i^k\|_{(i)\star} \geq \|\nabla_i f(X^k)\|_{(i)\star} - \|M_i^k - \nabla_i f(X^k)\|_{(i)\star}$, we can obtain 
	\begin{eqnarray*}
		f(X^{k+1}) &\leq& f(X) + \sum_{i=1}^p (2 + s(1-\delta_i)) t_i \eta  \| M_i^k - \nabla_i f(X^k)\|_{(i)\star}   -\sum_{i=1}^p s t_i \eta (1-\delta_i) \|\nabla_i f(X^k)\|_{(i)\star} \\ 
		&& + \sum_{i=1}^p \phi_i t_i \eta \epsilon_{1, i}   +  \sum_{i=1}^p  2t_i^2 \eta^2 (1 + (1+\delta_i)^2)(L_i^0 + L_i^1 \|\nabla_i f(X^k)\|_{(i)\star}) \\ 
		&\leq&  \beta f(X^*) + (1-\beta) f(X^k) + \sum_{i=1}^p  2t_i^2 \eta^2 (1 + (1+\delta_i)^2) L_i^0  +   \sum_{i=1}^p \phi_i t_i \eta \epsilon_{1, i} \\ 
		&& +  \sum_{i=1}^p (2 + s(1-\delta_i)) t_i \eta  \| M_i^k - \nabla_i f(X^k)\|_{(i)\star} \\ 
		&& - \sum_{i=1}^p t_i \eta \left(  s(1-\delta_i) - 2t_i \eta (1 + (1+\delta_i)^2)L_i^1  \right) \|\nabla_i f(X^k)\|_{(i)\star}, 
	\end{eqnarray*}
	where the last inequality comes from the star-convexity in Assumption \ref{as:starconvex}. 
	
	Combining the above inequality and (\ref{eq:Mik-nablaf-sc}), we arrive at 
	\begin{eqnarray*}
		&&\E [f(X^{k+1}) - f(X^*)]  \\ 
		&\leq& (1-\beta) \E [f(X^k) - f(X^*)]  +  \sum_{i=1}^p  2 L_i^0 t_i^2 \eta^2 (1 + (1+\delta_i)^2)   +  \sum_{i=1}^p \phi_i t_i \eta \epsilon_{1, i}   \\ 
		&& - \sum_{i=1}^p t_i \eta \left(  s(1-\delta_i) - 2t_i \eta (1 + (1+\delta_i)^2)L_i^1  \right) \E [\|\nabla_i f(X^k)\|_{(i)\star}] \\ 
		&& +  (2 + s(1-\delta_i))t_i \eta \left(   (1-\alpha)^k\rho \sigma + \sqrt{\alpha} \rho \sigma + \frac{2(1+\delta_i)L_i^0 t_i \eta}{\alpha}  \right)   \\ 
		&& +  \sum_{i=1}^p 2(2+2\delta_i + s(1-\delta_i^2))L_i^1 t_i^2 \eta^2 \sum_{\tau=1}^k (1-\alpha)^{k+1-\tau} \E [\|\nabla_i f(X^\tau)\|_{(i)\star}], 
	\end{eqnarray*}
	which implies that 
	\begin{eqnarray*}
		&& \E [f(X^K) - f(X^*)] \\ 
		&\leq& (1-\beta)^K (f(X^0) - f(X^*)) + \frac{1}{\beta}  \sum_{i=1}^p  2(1 + (1+\delta_i)^2)L_i^0 t_i^2 \eta^2  +  \sum_{i=1}^p \frac{\phi_i t_i \eta \epsilon_{1, i}}{\beta}  \\ 
		&&  +  \sum_{i=1}^p (2 + s(1-\delta_i))t_i \eta \left(  \frac{\rho\sigma}{\alpha}  +  \frac{\sqrt{\alpha} \rho \sigma}{\beta}  +   \frac{2(1+\delta_i)L_i^0 t_i \eta}{\alpha \beta}   \right) \\ 
		&& - \sum_{i=1}^p \sum_{j=0}^{K-1} (1-\beta)^{K-1-j} t_i \eta \left(  s(1-\delta_i) - 2t_i \eta (1 + (1+\delta_i)^2)L_i^1  \right) \E [\|\nabla_i f(X^j)\|_{(i)\star}] \\ 
		&&  +  \sum_{i=1}^p 10L_i^1 t_i^2 \eta^2 \sum_{j=0}^{K-1} (1-\beta)^{K-1-j} \sum_{\tau=1}^j (1-\alpha)^{j+1-\tau} \E [\|\nabla_i f(X^\tau)\|_{(i)\star}], 
	\end{eqnarray*}
	where we also use $2+2\delta_i + s(1-\delta_i^2) \leq 5$. 
	
	For $\alpha > \beta$, we have 
	\begin{eqnarray*}
		&& \sum_{j=0}^{K-1} (1-\beta)^{K-1-j} \sum_{\tau=1}^j (1-\alpha)^{j+1-\tau} \E [\|\nabla_i f(X^\tau)\|_{(i)\star}] \\ 
		&=& \sum_{\tau=1}^{K-1} \sum_{j=\tau}^{K-1} (1-\beta)^{K-1-j} (1-\alpha)^{j+1-\tau} \E [\|\nabla_i f(X^\tau)\|_{(i)\star}] \\ 
		&=&   \sum_{\tau=1}^{K-1} (1-\beta)^{K-\tau}   \E [\|\nabla_i f(X^\tau)\|_{(i)\star}] \sum_{j=\tau}^{K-1} \left(  \frac{1-\alpha}{1-\beta}  \right)^{j+1-\tau} \\ 
		&\leq& \frac{1-\beta}{\alpha-\beta}  \sum_{\tau=1}^{K-1} (1-\beta)^{K-\tau}   \E [\|\nabla_i f(X^\tau)\|_{(i)\star}] \\ 
		&=&  \frac{1-\beta}{\alpha-\beta}  \sum_{j=1}^{K-1} (1-\beta)^{K-j}   \E [\|\nabla_i f(X^j)\|_{(i)\star}]. 
	\end{eqnarray*}
	
	Then from the above two inequalities, under the assumption that $s(1-\delta_i) \geq 2t_i \eta L_i^1 (1 + (1+\delta_i)^2 + \frac{5}{\alpha - \beta})$ for all $i$, we can obtain 
	\begin{eqnarray*}
		&& \E [f(X^K) - f(X^*)] \\ 
		&\leq& (1-\beta)^K (f(X^0) - f(X^*)) + \frac{1}{\beta}  \sum_{i=1}^p  2(1 + (1+\delta_i)^2)L_i^0 t_i^2 \eta^2  +  \sum_{i=1}^p \frac{\phi_i t_i \eta \epsilon_{1, i}}{\beta}  \\ 
		&&  +   \sum_{i=1}^p (2 + s(1-\delta_i))t_i \eta \left(  \frac{\rho\sigma}{\alpha}  +  \frac{\sqrt{\alpha} \rho \sigma}{\beta}  +   \frac{2(1+\delta_i)L_i^0 t_i \eta}{\alpha \beta}   \right). 
	\end{eqnarray*}

\end{document}